\documentclass[10pt,a4paper]{amsart}

\usepackage{amsmath}
\usepackage{amsthm}
\usepackage{enumerate} 
\usepackage{amssymb}
\usepackage{graphicx}
\usepackage[all]{xy}
\usepackage{hyperref}
\usepackage{aliascnt}
\usepackage{xcolor}
\usepackage{stmaryrd} 
\usepackage{mathtools}
\usepackage{verbatim} 
\usepackage{txfonts} 
\usepackage{lmodern}
\usepackage{setspace}   

\usepackage[hmargin=2.9cm]{geometry}
\newtheorem{lma}{Lemma}[section]

\newaliascnt{thmCt}{lma}
\newtheorem{thm}[thmCt]{Theorem}
\aliascntresetthe{thmCt}

\newaliascnt{corCt}{lma}
\newtheorem{cor}[corCt]{Corollary}
\aliascntresetthe{corCt}

\newaliascnt{propCt}{lma}
\newtheorem{prop}[propCt]{Proposition}
\aliascntresetthe{propCt}

\newtheorem*{thm*}{Theorem}
\newtheorem*{cor*}{Corollary}
\newtheorem*{prop*}{Proposition}

\theoremstyle{definition}

\newaliascnt{prgCt}{lma}
\newtheorem{prg}[prgCt]{}
\aliascntresetthe{prgCt}

\newaliascnt{dfnCt}{lma}
\newtheorem{dfn}[dfnCt]{Definition}
\aliascntresetthe{dfnCt}

\newaliascnt{rmkCt}{lma}
\newtheorem{rmk}[rmkCt]{Remark}
\aliascntresetthe{rmkCt}

\newaliascnt{rmksCt}{lma}

\aliascntresetthe{rmksCt}

\newaliascnt{ntnCt}{lma}
\newtheorem{ntn}[ntnCt]{Notation}
\aliascntresetthe{ntnCt}

\newaliascnt{ntnsCt}{lma}

\aliascntresetthe{ntnsCt}

\newaliascnt{qstCt}{lma}
\newtheorem{qst}[qstCt]{Question}
\aliascntresetthe{qstCt}

\newaliascnt{prblCt}{lma}

\aliascntresetthe{prblCt}

\newaliascnt{obsCt}{lma}

\aliascntresetthe{obsCt}

\newaliascnt{exaCt}{lma}
\newtheorem{exa}[exaCt]{Example}
\aliascntresetthe{exaCt}

\newaliascnt{exasCt}{lma}

\aliascntresetthe{exasCt}

\newcommand{\T}{\mathbb{T}}
\newcommand{\C}{\mathbb{C}}
\newcommand{\N}{\mathbb{N}}

\newcommand{\R}{\mathbb{R}}
\newcommand{\Z}{\mathbb{Z}}
\newcommand{\K}{\mathrm{K}}

\newcommand{\HH}{\mathrm{H}}

\DeclareMathOperator{\Mat}{Mat}
\DeclareMathOperator{\im}{im}

\DeclareMathOperator{\Th}{Th}
\DeclareMathOperator{\diag}{diag}

\DeclareMathOperator{\spectrum}{sp}

\newcommand{\CatCa}{\mathrm{C}^*}

\DeclareMathOperator{\Aff}{Aff}

\DeclareMathOperator{\Lat}{Lat}

\DeclareMathOperator{\AbGp}{AbGp}

\DeclareMathOperator{\F}{F}
\DeclareMathOperator{\Cu}{Cu}

\DeclareMathOperator{\AI}{AI}

\DeclareMathOperator{\A}{A}

\DeclareMathOperator{\UHF}{UHF}
\DeclareMathOperator{\AH}{AH}

\DeclareMathOperator{\Lsc}{Lsc}

\DeclareMathOperator{\ev}{ev}

\DeclareMathOperator{\Tr}{Tr}
\DeclareMathOperator{\Hom}{Hom}
\DeclareMathOperator{\id}{id}

\DeclareMathAlphabet{\mymathbb}{U}{bbold}{m}{n}

\begin{document}
\onehalfspacing
\title{On the Nielsen-Thomsen sequence}

\author{Laurent Cantier}
\email{lncantier@gmail.com}
\address{Laurent Cantier,\newline
Departamento de Matem\'aticas\\
Universidad de Zaragoza\\
C/Pedro Cerbuna 12\\
50009 Zaragoza\\
Spain}
\urladdr{www.laurentcantier.fr}

\thanks{\textit{Email address}: lncantier@gmail.com\\
The author was supported by the Spanish Ministry of Universities and the European Union-NextGenerationEU through a Margarita Salas grant and partially supported by MINECO (grant No. PID2023-147110NB-I00), and by the Departament de Recerca i Universitats de la Generalitat de Catalunya (grant No. 2021-SGR-01015).}


\begin{abstract} 
The Nielsen-Thomsen sequence plays a pivotal role in refining invariants for $\CatCa$-algebras beyond the Elliott classification framework. This paper revisits the sequence, introducing the concepts of \emph{Nielsen-Thomsen bases}, \emph{rotation maps} and \emph{diagonalisable morphisms}, to better understand its unnatural splitting. These insights enable novel comparison methods for *-homomorphisms at the level of the Hausdorffized algebraic $\K_1$-groups, and subsequently the Hausdorffized unitary Cuntz group. 

We apply our methods to classification via the Hausdorffized unitary Cuntz semigroup. In particular, we present a new proof of the non-isomorphism between two $\A\!\T$-algebras constructed by Gong, Jiang and Li. We also exhibit several pairs of non-unitarily equivalent *-homomorphisms with domain $C(\T)$.
\end{abstract}

\maketitle

\section{Introduction}
The classification program for nuclear $\CatCa$-algebras has seen remarkable progress in the past decade, building on Elliott's groundbreaking conjecture that simple, separable, unital, nuclear $\CatCa$-algebras could be classified by $\K$-theoretic and tracial invariants. This vision was fully realized for the class of $\mathcal{Z}$-stable algebras satisfying the Universal Coefficient Theorem (UCT), through the combined efforts of many researchers, including the breakthrough works \cite{EGLN21} and \cite{GLN1,GLN2}. These developments culminated in the classification of simple, separable, unital, nuclear, $\mathcal{Z}$-stable $\CatCa$-algebras satisfying the UCT. See \cite{W18} for a general overview, and \cite{CGSTW21} for a remarkably detailed and innovative exposition on the matter. 
Nevertheless, classification outside the simple and $\mathcal{Z}$-stable setting presents additional challenges that have motivated the development of refined invariants. Already in early classification results, Nielsen and Thomsen presented a split-exact sequence, in the study and classification of *-homomorphisms between circle algebras. See in \cite{NT96}. This sequence, now referred to as the \emph{Nielsen-Thomsen sequence}, makes explicit the relationship between the Hausdorffized algebraic $\K_1$-group and the tracial state space, via the de la Harpe-Skandalis determinant. While their seminal work revealed its potential for distinguishing *-homomorphisms sharing identical Elliott invariants, several fundamental questions and applications, e.g. on the unnatural splitting of the sequence, have remained unexplored. 

The past ten years have seen significant advances in the classification of non-simple $\CatCa$-algebras. We recall Robert's classification of inductive limit of one-dimensional NCCW-complexes with trivial $\K_1$-groups, by means of the Cuntz semigroup. See \cite{R12}. Another approach began with classification of $\A\!\T$ algebras of real rank zero by means of the traditional Elliott invariant. See \cite{E93,EG96}. This second approach has been expanded over the years, resulting  in Gong, Jiang, and Li's classification of $\AH$ algebras with the ideal property, by means of a refined version of the Elliott invariant. See \cite{GJL20}. These works highlighted the crucial role played by the ideal structure and the Nielsen-Thomsen sequence for classification. Additionally, recent developments of the Cuntz semigroup and its refined unitary versions, have provided new perspectives on these challenges. See \cite{C23a,C23b,C25}.

In this paper, we explore the structural and internal properties of the Hausdorffized unitary Cuntz semigroup. Our investigation begins with a systematic study of the Nielsen-Thomsen sequence, to better understand the information encoded at the morphism level. To achieve this, we introduce new concepts such as Nielsen-Thomsen bases and rotation maps. These tools enable us to reinterpret the \textquoteleft unnatural\textquoteright\ splitting phenomenon, through matrix representations of *-homomorphisms at the level of the Hausdorffized $\K_1$-group. Our approach quantifies the effect of *-homomorphisms on given Nielsen-Thomsen bases, by measuring the disturbance caused by these maps. Consequently, we are able to construct  a metric $\mathfrak{d}$ to compare *-homomorphisms at the level of the Hausdorffized algebraic $\K_1$-group. While this framework provides a new perspective on distinguishing *-homomorphisms between $\CatCa$-algebras that agree on the traditional Elliott invariant, it also establishes, in turn, novel comparison methods for the Hausdorffized unitary Cuntz group and its morphisms. Finally, we apply our methods to distinguishing $\CatCa$-algebras and *-homomorphisms with domain $C(\T)$, by means of the Hausdorffized unitary Cuntz group. We gather all our results in the following theorem.

\begin{thm*}
(i) The (non-simple) $\A\!\T$-algebras constructed in \cite{GJL20} agree on the Elliott invariant, the Hausdorffized algebraic $\K_1$-group, the Cuntz semigroup and its unitary version. 

Yet, they are distinguished by the Hausdorffized unitary Cuntz semigroup.\\

(ii) There exists a family $\{\varphi_k\colon C(\T)\longrightarrow C[0,1]\otimes M_{2^\infty}\}_{k\in\N}$ of *-homomorphisms agreeing on the Elliott invariant, the Cuntz semigroup and its unitary version. 

Yet, we compute that $\mathfrak{d}(\overline{\K}_1(\varphi_k),\overline{\K}_1(\varphi_l))=\frac{\vert k-l\vert}{2}$ and hence, they are pairwise distinguished by the Hausdorffized algebraic $\K_1$-group.\\

(iii) There exist two *-homomorphisms $\varphi_u,\varphi_v\colon C(\T)\longrightarrow A$, where $A$ is an $\AI$-algebra, agreeing on the Elliott invariant, the Cuntz semigroup and its unitary version and the $\overline{\K}_1$-group.

Yet, they are distinguished by the Hausdorffized unitary Cuntz semigroup.
\end{thm*}

\textbf{Organization of the paper.}
The paper is organized as follows. Section 2 is devoted to variations of Cuntz semigroups and the comparison theory of their morphisms, building on recent developments in the theory of $\Cu$-semigroups. In Section 3, we introduce the necessary background on Nielsen-Thomsen bases and develop the theory of rotation maps for *-homomorphisms. This leads to the construction of a relevant metric to compare *-homomorphisms at the level of their Hausdorffized algebraic unitary group. Finally, in Section 4, we demonstrate the effectiveness of our methods by providing alternative proofs of the aforementioned classification results, illustrating how our framework naturally captures and explains these phenomena through the lens of the Nielsen-Thomsen sequence and its associated invariants.

\textbf{Acknowledgments.} The author would like to thank B. Jacelon for inspiring discussions around unitary elements of $\CatCa$-algebras and the de la Harpe-Skandalis determinant. Also, he is indebted to the referees that have helped to improve the quality and readability of the paper.

\section{Variations of the Cuntz semigroup}
\label{sec:Sec2}

\subsection{The Cuntz semigroup and its refined versions.}
The Cuntz semigroup was introduced in \cite{C78} and has been a powerful tool for classification since the milestone paper \cite{CEI08}. We refer the reader to \cite{GP26} for a detailed survey around the Cuntz semigroup, where they shall find all the basics they might need, and to \cite{APRT22} for state of the art results.

Let $(S,\leq)$ be an ordered monoid and let $x,y$ in $S$. We say that $x$ is \emph{way-below} $y$, and we write $x\ll y$, if for all increasing sequences $(z_n)_{n\in\N}$ in $S$, with supremum such that $\sup\limits_{n\in\N} z_n\geq y$, then there exists $k$ such that $z_k\geq x$. 

We say that a (positively) ordered monoid $S$ is an \emph{abstract Cuntz semigroup}, or a $\Cu$-semigroup, if $S$ satisfies the below axioms.

(O1) Every increasing sequence of elements in $S$ has a supremum. 

(O2) For any $x\in S$, there exists a $\ll$-increasing sequence $(x_n)_{n\in\N}$ in $S$ such that $\sup_{n} x_n= x$.

(O3) Addition and the compact containment relation are compatible.

(O4) Addition and suprema of increasing sequences are compatible.

We say that a map $\alpha:S\longrightarrow T$ is a $\Cu$-morphism, if $\alpha$ is an ordered monoid morphism preserving the compact-containment relation and suprema of increasing sequences.

The category $\Cu$ consists of $\Cu$-semigroups and $\Cu$-morphisms. Note that it is commonly assumed for a $\Cu$-semigroup to be positively ordered, which we do not assume here. 

\begin{dfn}[The Cuntz semigroup of a $\CatCa$-algebra] Let $A$ be a $\CatCa$-algebra and let $a,b\in A_+$. 

We say that $a$ \emph{Cuntz subequivalent} to $b$ in $A$, denoted $a \lesssim_{\Cu} b$, if for every $\epsilon > 0$,  there exists $ r\in A$ such that $\left\lVert rbr^*-a \right\rVert < \epsilon$. We antisymmetrize this subequivalence to obtain a an equivalence relation $\sim_{\Cu}$ called the \emph{Cuntz equivalence relation}.
The \emph{Cuntz semigroup} of $A$ is 
  \[
  \Cu(A):=((A \otimes \mathcal{K})_{+}/\sim_{\Cu},+,\leq)
  \]
where the addition is canonically defined and the order is induced by $\precsim_{\Cu}$.
\end{dfn}

It has been shown in \cite{CEI08} and \cite{APT18} that the functor $\Cu\colon \CatCa\longrightarrow \Cu$ is well-defined and continuous. 

\begin{prg}\textbf{The $\Cu_\K$-construction.} In \cite{C23b}, the author proposed a systematic construction for incorporating additional invariants into the Cuntz semigroup, resulting in refined versions that capture more structural information about $\CatCa$-algebras and their ideals. By leveraging categorical properties of the Cuntz semigroup and its target category $\Cu$, these refined invariants provide a robust context for classification. Here, we recall a process yielding two unitary variants of the Cuntz semigroup, and we refer the reader to \cite[Section 4.A]{C23b} for more details.

Let $A$ be $\CatCa$-algebra and let $\K\colon \CatCa\rightarrow \AbGp$ be a continuous functor. We view $\Cu(A)$ as a partially ordered set which induces a small category. (The objects are its elements and the morphisms are induced by the order.) Therefore, we can consider the functor 
\[
\begin{array}{ll}
	{\K_A}: \Cu(A)\longrightarrow \AbGp \\
	\hspace{1,6cm} x\longmapsto \K(I_x)\\
	\hspace{0,95cm}x\leq y\longmapsto \K(I_x\overset{\subseteq}\longrightarrow I_y)
\end{array}
\] 
where $I_x\in\Lat_f(A)$ is obtained through the isomorphism $\Lat_f(\Cu(A))\simeq \Lat_f(A)$. (See \cite[Proposition 5.1.10]{APT18} for the latter morphism.) To ease notations, we may write $\K(x)$ instead of $\K_A(x)$.

The \emph{$\Cu_\K$-construction of a $\CatCa$-algebra $A$} is an ordered monoid $\Cu_\K(A)$ consisting of pairs $(x,g)$, where $x\in \Cu(A)$ and $g\in \K(I_x)$. 
The addition and order are respectively given by 
\[
\left\{\begin{array}{ll}(x,g)+(y,h):=(x+y, \K(x\leq x+y)(g)+ \K(y\leq x+y)(h)) \text{ and}\\
(x,g)\leq(y,h), \text{ whenever } x\leq y \text{ in} \Cu(A) \text{ and }\K(x\leq y)(g)=h \text{ in } \K(y).  
\end{array}\right.
\]

The \emph{$\Cu_\K$-construction of a *-homomorphism $\phi\colon A\rightarrow B$} is an ordered monoid morphism $\Cu_\K(\phi)$ given by 
\[\Cu_\K(\phi)\colon(x,g)\mapsto (\Cu(\phi)(x),\K(I_x\overset{\phi}{\longrightarrow} ({I_x})_\phi)(g))\]
where $I_\phi\in \Lat(B)$ denotes the smallest ideal of $B$ containing $\phi(I)$. Therefore, $\Cu_\K(\phi)$ is entirely determined by the data $(\Cu(\phi),\{\overline{\K}_1(I\overset{\phi}{\longrightarrow} I_\phi)\}_{I\in\Lat_f(A)})$.
\end{prg}

We briefly recall some definitions and properties about $\Cu_\K$-constructions that will be of use later. Let us briefly mention that below, the category $\Cu^*$ is a subcategory of $\Cu$, except for the ordered monoids involved (which we recall, need not be positively ordered) satisfy extra axioms ensuring enough \textquoteleft positivity\textquoteright\ in the monoid. In particular any positively ordered $\Cu$-semigroup belongs to $\Cu^*$. It is of limited use to specify more on the matter and  we refer the reader to \cite{C23b} for details. 

\begin{thm}[{\cite[Theorems 4.2 - 4.5]{C23b}}]\label{thm:exactCUK} Let $\K\colon \CatCa\rightarrow \AbGp$ be a continuous functor.

(i) The assignment $\Cu_{\K}\colon \CatCa\longrightarrow \Cu^*$ is a well-defined continuous functor. 

(ii) Let $\phi\colon A\rightarrow B$ be a *-homomorphism between $\CatCa$-algebras and let  $I\in\Lat_f(A)$. 
Then for any $J\in\Lat(B)$ such that $J\supseteq I_\phi$, the following diagram is commutative with exact rows in $\Cu^*$
\[
D(I,\phi,J)\colon\xymatrix
{
0\ar[r]^{} & \Cu(I)\ar[d]_{\Cu(\phi)}\ar[r]^{i_1} & \Cu_{\K}(I) \ar[d] _{\Cu_{\K}(\phi)}\ar[r]^{j_1} & \K(I)\ar@/_{-1,2pc}/[l]^{q_1} \ar[d]^{\K(I\overset{\phi}{\rightarrow} I_\phi\subseteq J)}\ar[r]^{} & 0
\\
0\ar[r]^{} & \Cu(J)\ar[r]^{i_2} & \Cu_{\K}(J)\ar[r]^{j_2} & \K(J)\ar@/_{-1,2pc}/[l]^{q_2}\ar[r]^{} & 0
} 
\]
where $I_\phi$ is the smallest element of $\Lat(B)$ obtained by $I$ through $\phi$, i.e. $I_\phi:=\overline{B\phi(I)B}$.

(iii) The canonical diagram  $D(I,\phi,J)\overset{\iota}{\longrightarrow}D(A,\phi,B)$ is also commutative with exact rows.
\end{thm}

Lastly, let us introduce a couple of refined versions of the Cuntz semigroup obtained via the $\Cu_\K$-construction, applied to the unitary group $\K_1\colon \CatCa\rightarrow \AbGp$ and its Hausdorffized refinement $\overline{\K}_1\colon \CatCa\rightarrow \AbGp$. We refer the reader to \cite{C21a} and \cite[4.B - 4.11]{C23b} for more on these invariants.

\begin{dfn}[The unitary variants of the Cuntz semigroups of a $\CatCa$-algebra]\label{dfn:unitaryCu}  Let $A$ be a $\CatCa$-algebra. The above construction yields the following.

(i) The \emph{unitary Cuntz semigroup} of $A$, denoted by $\Cu_{\K_1}(A)$, is the merging of the Cuntz semigroup with the $\K_1$-group. 

(ii) The \emph{Hausdorffized unitary Cuntz semigroup} of $A$, denoted by $\Cu_{\overline{\K}_1}(A)$, is the merging of the Cuntz semigroup with the Hausdorffized algebraic $\K_1$-group.
\end{dfn}

We remark that the unitary Cuntz semigroup has been introduced and studied intensively by the author, and therebefore termed $\Cu_1$ instead of $\Cu_{\K_1}$. See \cite{C21a,C21b}. Originally, this invariant was intended to relax the trivial $\K_1$-assumption off of the classification achieved in \cite{R12}. While \cite{C23a} provided preliminary examples in this direction, \cite{C25} established that a Hausdorffized version was necessary, in hope of achieving a complete classification.

\begin{prg}\textbf{The category $\Cu_\mathfrak{w}$.}
The general procedure of merging a $\Cu$-semigroup $S$ with a functor $G\colon S\rightarrow \AbGp$ has been introduced in \cite{C23b} and referred to as a \textquoteleft webbing transformation\textquoteright. Roughly speaking, given $S,T$ and $\alpha_0\colon S\rightarrow T$ in the category $\Cu$, together with functors $G\colon S\rightarrow \AbGp,H\colon T\rightarrow \AbGp$ and a natural transformation $\eta\colon G\Rightarrow H\circ \alpha_0$, we obtain objects and morphisms in the category $\Cu^*$ in the following form
\[
S_G,T_H \quad\quad\text{ and }\quad\quad \alpha=(\alpha_0,\{\alpha_s\}_{s\in S})\colon S_G \rightarrow T_H
\]
where $\alpha_{s}\colon G(s)\rightarrow H(\alpha_0(s))$ is given by the natural transformation.

The webbing transformations form a subcategory $\Cu_\mathfrak{w}\subseteq\Cu^*\subseteq \Cu$ closed under inductive limits, which need not be full. Nevertheless, having classification in mind, it may be more adequate to consider the restriction $\Cu_\K\colon \CatCa\longrightarrow \Cu_\mathfrak{w}$ which is again a continuous functor, since any $\Cu_\K$-construction (of both $\CatCa$-algebras and *-homomorphisms) belongs to $\Cu_\mathfrak{w}$.
\end{prg}

\subsection{Comparison of Cu-morphisms and their refined versions}\label{subsection:2B} Classification of $\CatCa$-algebras may follow from classification of *-homomorphisms. In practice, this often requires methods to compare morphisms within the target category. 
In this part, we recall several comparison methods for $\Cu$-morphisms, and build upon them to exhibit such comparison methods for $\Cu_\mathfrak{w}$-morphisms between $\Cu_\K$-constructions. We highlight that all the subsequent can be done for broader webbing transformations and their morphisms, i.e. for the category $\Cu_\mathfrak{w}$.

Let $A,B$ be $\CatCa$-algebras and $\K\colon \CatCa\rightarrow \AbGp$ be a continuous functor. We consider  $\alpha:=(\alpha_0,\{\alpha_I\}_{I\in \Lat_f(A)})$ and $\beta:=(\beta_0,\{\beta_I\}_{I\in \Lat_f(B)})\colon \Cu_{\K}(A)\longrightarrow \Cu_{\K}(B)$ be $\Cu_\mathfrak{\K}$-morphisms. Note that $\alpha_0,\beta_0\colon \Cu(A)\rightarrow \Cu(B)$ are $\Cu$-morphisms and $\alpha_{I_x}\colon \K(x)\rightarrow \K(\alpha_0(x)), \beta_{I_x}\colon \K(x)\rightarrow \K(\beta_0(x))$ are group morphisms. We will refer to the above as the \emph{current setting}.

\begin{prg}\textbf{Finite-set comparison.} The first method is based on a finite-set comparison for $\Cu$-morphisms that we recall now. We refer the reader to \cite{C22,CV24} for more details. 

\begin{dfn}
Let $S,T$ be $\Cu$-semigroups and let $\alpha,\beta\colon S\rightarrow T$ be $\Cu$-morphisms. Let $F\subseteq S$ be a (finite) set. We say that $\alpha$ and $\beta$ \emph{compare on $F$}, and we write $\alpha\simeq_F\beta$,  if for any $f,g\in F$ such that $f\ll g$ (in $S$), we have that $\alpha(f)\leq \beta(g)$ and $\beta(f)\leq \alpha(g)$.
\end{dfn}

Observe that $\alpha=\beta$ if and only if $\alpha\simeq_F \beta$ for any finite set $F\subseteq S$ if and only if $\alpha\simeq_{\{f,g \mid f\ll g\}}\beta$. We now introduce the notion of fiber diagrams in the context of our current setting, allowing to make the finite-set comparison method specific for $\Cu_\mathfrak{w}$-morphisms between $\Cu_\K$-constructions.

\begin{dfn}[The fiber diagrams] Retain the current setting. Let $x,y\in \Cu(A)$ be such that $x\ll y$. We define the \emph{fiber diagram of $\{\alpha,\beta\}$ at coordinates $(x,y)$} as follows
\[
\xymatrix@!R@!C@R=15pt@C=30pt{
&\K(\alpha_0(y))\\
\K(\alpha_0(x))\ar[ur]^{\K(\leq)}\ar[dr]_{\K_\delta(\leq)}&\K(x)\ar[r]^{\!\!\!\!\!\!\!\!\!\!\beta_{I_x}}\ar[l]_{\quad\quad\alpha_{I_x}}&\K(\beta_0(x))\ar[dl]^{\K(\leq)}\ar[ul]_{\K_\delta(\leq)}
 \\
&\K(\beta_0(y))}  
\]
where the maps $\K_\delta(t\leq t')$ are defined to be $\K(t\leq t')$ for any $t,t'\in \Cu(B)$ such that $t\leq t'$, and the empty morphism (i.e., no arrows) otherwise.

We denote the above diagram by $\mathcal{F}_{\{\alpha,\beta\}}(x,y)$, or simply $\mathcal{F}(x,y)$, when the context is clear.
\end{dfn}

\begin{rmk}
The compositions $\K(\leq)\circ\alpha_{I_x}\colon \K(x)\rightarrow \K(\alpha_0(y))$ and $\K(\leq)\circ\beta_{I_x}\colon \K(x)\rightarrow \K(\beta_0(y))$ factorize through $\K(I_y)$. Hence the fiber diagram $\mathcal{F}(x,y)$ is equivalent to
\[
\xymatrix@!R@!C@R=25pt@C=30pt{
\K(\beta_0(x))\ar[d]_{\K_\delta(\leq)} & \K(x)\ar[l]_{\beta_{I_x}}\ar[r]^{\alpha_{I_x}}\ar[d]_{\K(\leq)} & \K(\alpha_0(x))\ar[d]^{\K_\delta(\leq)}\\
\K(\alpha_0(y)) & \K(y)\ar[l]^{\alpha_{I_y}}\ar[r]_{\beta_{I_y}} & \K(\beta_0(y))\\
}
\]
\end{rmk}

\begin{prop} Retain the current setting. Let $F\subseteq \Cu_\K(A)$ be a (finite) set. Let $F_0:=\{x\in \Cu(A)\mid (x,g)\in F, \text{ for some } g\in \K(x)\}$.

Then $\alpha\simeq_F\beta$ if and only if both of the the following conditions are satisfied.

(i) $\alpha_0\simeq_{F_0}\beta_0$

(ii) For any $x,y\in F_0$ such that $x\ll y$, and for any $g\in \K(x)$ such that $(x,g)\in F$, we have that 
\[
\left\{
\begin{array}{ll}
\K(\alpha_0(x)\leq\alpha_0(y))\circ\alpha_{I_x}(g)=\K(\beta_0(x)\leq\alpha_0(y))\circ\beta_{I_x}(g) \text{ in } \K(\alpha_0(y))\\
\K(\beta_0(x)\leq\beta_0(y))\circ\beta_{I_x}(g)=\K(\alpha_0(x)\leq\beta_0(y))\circ\alpha_{I_x}(g) \text{ in } \K(\beta_0(y))
\end{array}
\right..
\]
\end{prop}

\begin{proof}
This follows from the characterisation obtained \cite[Proposition 2.12]{C23b}, stating that $(x,g)\ll (y,h)$ if and only if $x\ll y \in S$ and $\K(x\leq y)(g)=h$. 
\end{proof}
\end{prg}

\begin{prg}\textbf{Metric comparison.} The second method is based on a metric comparison for $\Cu$-morphisms, sometimes referred to as \emph{$\Cu$-metrics}. 
Recall that whenever $X$ is a locally compact Hausdorff spaces of dimension at most one, it is known that $\Cu(C_0(X))\simeq \Lsc(X,\overline{\N})$. See \cite[Theorem 1]{R13}. More generally, it is known that $\Lsc(X,\overline{\N})$ is a $\Cu$-semigroup, for any compact Hausdorff space $X$, although it may differ from $\Cu(C(X))$. See \cite[Corollary~4.22]{V22} and \cite[Theorem~5.17]{APS11}. See also \cite{APS11,C22,C23a, R13} for more examples of Cuntz semigroups expressed as lower-semicontinuous functions.
 
For such $\Cu$-semigroups, we can construct the following $\Cu$-metric, which has been inspired by the works of \cite{CE08} and \cite{RS10}. We mention \cite{C22,C23a,C25} for an explicit use of this $\Cu$-metric, and \cite{C22,CV24} for more on $\Cu$-metrics. 
\begin{dfn}
Let $X$ be a (locally) compact Hausdorff space and let $T\in \Cu$. Consider $\Cu$-morphisms $\alpha, \beta\colon\Lsc(X,\overline{\N})\longrightarrow T$. We set 
\[  d_{\Cu}(\alpha,\beta):=\inf \Bigl\{ r>0\mid \forall  U\in\mathcal{O}(X),\, \alpha(\mymathbb{1}_{U})\leq\beta(\mymathbb{1}_{U_{r}})  \text{ and }  \beta(\mymathbb{1}_{U})\leq\alpha(\mymathbb{1}_{U_{r}})\Bigr\}
\] 
where $\mathcal{O}(X):=\{\text{Open sets of }X\}$ and $U_r:=\underset{x\in U}{\cup}B_r(x)$. If the infimum does not exist, we set the value to $\infty$.

This defines a metric on the set $\Hom_{\Cu}(\Lsc(X,\overline{\N}),T)$, referred to as the \emph{$\Cu$-metric}.
\end{dfn}

Let us retain the current setting. By a \emph{metric for the fiber diagrams of $(A,B)$}, we mean a metric $d$ on the set of fiber diagrams of $\{\alpha,\beta\}$, where $\alpha,\beta\in\Hom_{\Cu_\mathfrak{w}}(\Cu_\K(A),\Cu_\K(B))$. (Such metric always exists, e.g., take $d$ to be the trivial metric of the category $\AbGp$.)

\begin{dfn}[Norm of a fiber diagram] Retain the current setting. Let $d$ be a metric for the fiber diagrams of $(A,B)$. The \emph{norm of a fiber diagram $\mathcal{F}(x,y)$} associated to $d$ is defined by
\[
\Vert \mathcal{F}(x,y)\Vert_d:=\max_{(c,c')\in \Gamma}\{d(c,c')\}
\]
where $\Gamma$ is the set of pairs of composed morphisms in $\mathcal{F}(x,y)$ with same domain and codomain.
\end{dfn}

We are now ready to define several metrics to compare $\alpha$ and $\beta$ as follows. 

\begin{dfn}\label{dfn:metricK} Retain the current setting and assume that $\Cu(A)\simeq\Lsc(X,\overline{\N})$, for some locally compact Hausdorff space $X$. Let $d$ be a metric for the fiber diagrams of $(A,B)$. We consider
\[
d^*_{\Cu,d}(\alpha,\beta):=\inf_{} \Bigl\{r>\epsilon_0\mid \forall U\in \mathcal{O}(X), \, \Vert  \mathcal{F}(\mymathbb{1}_U,\mymathbb{1}_{U_r})\Vert_d\leq 4r  \Bigr\}
\]
where $\epsilon_0:=d_{\Cu}(\alpha_0,\beta_0)$.

This is a well-defined metric on the set $\Hom_{\Cu_{\mathfrak{w}}}(\Cu_\K(A),\Cu_\K(B))$. 
\end{dfn}

We can always choose $d=d_{\mathrm{triv}}$ to be the trivial metric, -i.e., the metric taking value $0$ whenever the morphisms are equal, and $\infty$ otherwise-. In that case, we may write $d^*_{\Cu}$ instead of $d^*_{\Cu,d_{\mathrm{triv}}}$. Also, the \textquoteleft factor 4\textquoteright\ appearing, is mostly for practical reasons, and is independent of the choice of $d$. See computations of \autoref{sec:Sec5}. 
Lastly, one could imagine a metric where the measures of the \textquoteleft$\Cu$-components\textquoteright\ and \textquoteleft the fiber diagrams components\textquoteright\ are uncorrelated but we do not pursue this idea here. We end the section with some useful computations.

\begin{prop}\label{prop:lwrbnd} Retain the current setting and assume that $\Cu(A)\simeq\Lsc(X,\overline{\N})$, for some locally compact Hausdorff space $X$. Let $d$ be a metric for the fiber diagrams of $(A,B)$.

(i) We have that $d_{\Cu}(\alpha_0,\beta_0)\leq d^*_{\Cu,d}(\alpha,\beta)$.

(ii) For any $x\in \Cu(A)$ and any $z\in \Cu(B)$ with $\alpha_0(x),\beta_0(x)\leq z$, we have that 
\[d(\K(\alpha_0(x)\leq z)\circ\alpha_I,\K(\beta_0(x)\leq z)\circ\beta_I)\leq 4d^*_{\Cu,d}(\alpha,\beta).\]  

(ii) Whenever $d:=d_{\mathrm{triv}}$ is the trivial metric, we have that
\[
d^*_{\Cu}(\alpha,\beta)=\inf_{} \Bigl\{r>\epsilon_0\mid \forall U\in \mathcal{O}(X),\,   \mathcal{F}(\mymathbb{1}_U,\mymathbb{1}_{U_r}) \text{ commutes}   \Bigr\}
\]
where $\epsilon_0:=d_{\Cu}(\alpha_0,\beta_0)$. 
\end{prop}

\begin{proof} All computations follow from construction and are left for the reader to check.
\end{proof}

\end{prg}

\section{The Rotation map of *-homomorphisms}
\label{sec:Sec3}
In \cite{NT96}, Nielsen and Thomsen pointed out that the Hausdorffized algebraic $\K_1$-group contains information unseen by the Elliott invariant, and hence, is essential for classifying *-homomorphisms. They observed that this group splits between the $\K_1$-group and traces together with projections,  giving rise to the \emph{Nielsen-Thomsen sequence}. Notably, they showed that this sequence splits \textquoteleft unnaturally\textquoteright\ at the morphism level. In this section, we delve deeper into this phenomenon by introducing the notions of Nielsen-Thomsen bases, rotation maps and diagonalisable morphisms. As a consequence, we are then able to quantify how *-homomorphisms affect the splitting behavior via metrics that we construct. 

\subsection{Nielsen-Thomsen bases and the rotation map} Before diving into the matter, let us recall preliminary facts around the de la Harpe-Skandalis determinant and its link with the Nielsen-Thomsen sequence. Relevant references are \cite{dHS84,NT96,T95}. See also \cite[2.2]{S24}.

\begin{prg}\textbf{The de la Harpe-Skandalis determinant.}
For a unital $\CatCa$-algebra $A$, we denote the subgroup of $A$ generated by additive commutators $[a,b]:=ab-ba$ by $[A,A]$ and we define the \emph{universal trace on $A$} to be the quotient map $\Tr:A\longrightarrow A/\overline{[A,A]}$. We extend this (continuous) tracial linear map to $M_{\infty}(A)$ in a canonical way that we also denote $\Tr_A:M_{\infty}(A)\longrightarrow A/\overline{[A,A]}$. 

Following the seminal work of \cite{dHS84}, for any piece-wise smooth path $\xi:[0,1]\longrightarrow U_0^{\infty}(A)$, we define
\[
\widetilde{\Delta}_{}(\xi):=\frac{1}{2i\pi}\int_0^1 \Tr(\xi'(t)\xi^{-1}(t))dt.
\]
It is shown in \cite[Lemme 1]{dHS84}, that $\widetilde{\Delta}$ only depends on the homotopy class of $\xi$ and that its image belongs to the quotient $A_{\mathrm{sa}}/A_0$, where $A_0:=\{h\in A_{\mathrm{sa}} \mid \Tr_A(h)=0 \}$. On the other hand, it is well-known that $A_{\mathrm{sa}}/A_0$ can be identified with $\Aff T_1(A)$ by sending $[h]$ to $\mathrm{ev}_h\colon \tau\in T_1(A)\mapsto \tau(h)\in \R$, where $T_1(A)$ denotes the (bounded) tracial states of $A$.

Therefore, the application $\widetilde{\Delta}_{}$ naturally induces a group morphism $\underline{\Tr}_A:\K_0(A)\longrightarrow \Aff T_1(A)$ by sending $[p]\mapsto \mathrm{ev}_p$, for any projection $p\in M_{\infty}(A)$. (This is done via the identification $\K_0(A)\simeq \pi_1(U_0^{\infty}(A))$.) We now define the \emph{de la Harpe-Skandalis determinant} as
\[
\begin{array}{ll}
\Delta_{}:U_0^{\infty}(A)\longrightarrow \Aff T_1(A)/\im(\underline{\Tr}_A)\\
\hspace{1,55cm} u\longmapsto [\widetilde{\Delta}_{}(\xi_u)]
\end{array}
\]
where $\xi_u$ is any piece-wise smooth path connecting the identity to $u$. In particular, $\Delta_{}(e^{2i\pi h})=[\mathrm{ev}_h]$ for any self-adjoint $h\in M_{\infty}(A)$. 

Even though it was observed in the original work de la Harpe and Skandalis, that $\overline{\ker \Delta_{}}\subseteq \overline{DU_0^\infty(A)}$ (see \cite[Proposition 4]{dHS84}), there exist cases where  $DU_0^\infty(A)\subsetneq \ker \Delta\subsetneq \overline{DU_0^\infty(A)}$. A fortiori, the group morphism $\Delta$ is not continuous in general. However, Thomsen introduced a non-stable version of the determinant to obtain a continuous version. Let $\overline{\im (\underline{\Tr}_A)}$ denote the closure of the image of $\underline{\Tr}_A$ in $\Aff T_1(A)$. We now define the  \emph{non-stable de la Harpe-Skandalis determinant}
\[
\begin{array}{ll}
\overline{\Delta}_{}:U_0^{\infty}(A)\longrightarrow \Aff T_1(A)/\overline{\im (\underline{\Tr}_A)}\\
\hspace{1,78cm} u\longmapsto [\Delta_{}(u)]
\end{array}
\]
to be the composition of $\Delta_{}$ with the  projection of $\Aff T_1(A)/\im (\underline{\Tr}_A)\ \relbar\joinrel\twoheadrightarrow \Aff T_1(A)/\overline{\im (\underline{\Tr}_A)}$. It is shown in \cite[Theorem 3.2]{T95} that $\overline{\Delta}_{}$ is a surjective continuous group homomorphism satisfying that 
\[
U_0^\infty(A)/\overline{DU_0^\infty(A)}\overset{\overline{\Delta}_{}}{\simeq} \Aff T_1(A)/\overline{\im (\underline{\Tr}_A)}.
\]

For convenience, we may write $\mathrm{H}(A):=\Aff T_1(A)/\overline{\im (\underline{\Tr}_A)}$.\\
\end{prg}

\begin{prg}\textbf{The Nielsen-Thomsen sequence.}
Let $A$ be a unital $\CatCa$-algebra. It is known that $DU(A)\subseteq DU_0(M_3(A))\subseteq U^\infty_0(A)$. We deduce that $\overline{DU^\infty(A)}$ is a normal subgroup of $U^\infty_0(A)$ and we obtain a canonical extension 
\[
0\longrightarrow U^\infty_0(A)/\overline{DU^\infty(A)}\overset{i}{\longrightarrow} U^\infty(A)/\overline{DU^\infty(A)}\overset{\pi}{\longrightarrow} U^\infty(A)/\overline{U^\infty_0(A)}\longrightarrow 0
\]

The middle term is often referred to as the \emph{Hausdorffized algebraic $\K_1$-group} that we will denote by $\overline{\K}_1$ to ease the notations.

It is readily checked that $U^\infty_0(A)/\overline{DU^\infty(A)}$ is a divisible subgroup of $U^\infty(A)/\overline{DU^\infty(A)}$. Therefore, the latter extension splits.
Finally, by using the non-stable de la Harpe-Skandalis determinant, we obtain the extension referred to as the \emph{Nielsen-Thomsen sequence}
\[
0\longrightarrow \mathrm{H}(A)\overset{i}{\longrightarrow} \overline{\K}_1(A)\overset{\pi}{\longrightarrow}\K_1(A)\longrightarrow 0
\]
where $\mathrm{H}(A):=\Aff T_1(A)/\overline{\im (\underline{\Tr}_A)}$.

It is worth mentioning that under the stable rank one assumption, we may work in the $\CatCa$-algebra itself instead of matrices, since we have that $\mathrm{H}(A)\simeq U_0(A)/\overline{DU(A)}$, $\overline{\K}_1(A)\simeq U(A)/\overline{DU(A)}$ and $\K_1(A)\simeq U(A)/U_0(A)$. See \cite[Theorem 2.10]{R87} and \cite[Proposition 4]{LR95} together with \cite[Corollary 3.4]{T95}. See also \cite[Proposition 2.5]{C25}.

We have seen that the Hausdorffized algebraic $\K_1$-group of any $\CatCa$-algebra $A$ is entirely recovered by the original Elliott invariant, since we have an isomorphism $\overline{\K}_1(A)\simeq \mathrm{H}(A)\oplus \K_1(A)$.
  Nevertheless, Nielsen and Thomsen showed in \cite{NT96} that the sequence splits \textquoteleft unnaturally\textquoteright\, at the level of morphisms. In other words, given a *-homomorphism $\phi\colon A\longrightarrow B$, the data of $\mathrm{H}(\phi)$ and $\K_1(\phi)$, where $\mathrm{H}(\phi)\colon \Aff T_1(A)/\overline{\im (\underline{\Tr}_A)}\longrightarrow \Aff T_1(B)/\overline{\im (\underline{\Tr}_B)}$ is the induced morphism by $\phi$, is (in general) not enough to recover $\overline{\K}_1(\phi)$. We can restate the above as follows.

(i) For any $\CatCa$-algebra $A$, the group $\overline{\K}_1(A)$ is split. 

(ii) For any *-homomorphism $\phi\colon A\longrightarrow B$, the group morphism $\overline{\K}_1(\phi)$ need not be diagonalisable. 

As a conclusion, it is necessary to incorporate the functor $\overline{\K}_1$, to classify a larger class of *-homomorphisms and $\CatCa$-algebras, that cannot be classified via the original Elliott invariant. 
\end{prg}

\begin{prg}We shed new light on the unnaturalness of the Nielsen-Thomsen sequence through \emph{Nielsen-Thomsen bases}. These bases provide a systematic way to analyse the splitting behavior of the sequence at the morphism level. We are able to explicitly compute a matrix representation of *-homomorphisms at the level of the $\overline{\K}_1$-group in terms of their value on traces, $\K_1$-group and what we refer to as \emph{rotation maps with respect to the bases}. This approach not only clarifies why the sequence splits unnaturally, but also allows us to quantify and interpret the deviations caused by *-homomorphisms. 
\end{prg}

\begin{ntn} For any unitary element $u$ of a $\CatCa$-algebra $A$, we shall denote its respective unitary equivalence classes by $\overline{u}:=[u]_{\overline{\K}_1(A)}$ and $[u]:=[u]_{\K_1(A)}$.
\end{ntn}
\begin{dfn}[Nielsen-Thomsen bases]
Let $A$ be a unital $\CatCa$-algebra. A \emph{Nielsen-Thomsen basis of $A$} is a set $\{c_k \in U^\infty(A)\}_{k\in \K_1(A)}$ of elements in $U^\infty(A)$ indexed by $\K_1(A)$ such that the mapping 
\[
\begin{array}{ll}
s_A\colon \K_1(A)\longrightarrow \overline{\K}_1(A)\\
\hspace{1,45cm} k\longmapsto \overline{c_k}
\end{array}
\]
is a group morphism satisfying $\pi_A\circ s_A=\id_{\K_1(A)}$. By convention, we always fix $c_0:=1_A.$
\end{dfn}
The next proposition follows from standard results regarding split-extensions of abelian groups.
\begin{prop} Let $A$ be a unital $\CatCa$-algebra. Then

(i) $A$ admits Nielsen-Thomsen bases.

(ii) Any Nielsen-Thomsen basis $\mathcal{C}:=\{c_k\}_{k\in\K_1(A)}$ of $A$ induces a unique group isomorphism $\nu_\mathcal{C}\colon \overline{\K}_1(A)\simeq \mathrm{H}(A)\oplus\K_1(A)$ sending $\overline{u}\mapsto (\overline{\Delta}_{}(
\begin{smallmatrix} u & 0 \\
0 & c_{[u]}^*
\end{smallmatrix}
),[u])$. 

(ii') Any Nielsen-Thomsen basis of $A$ induces a unique retract $r\colon \overline{\K}_1(A)\longrightarrow \mathrm{H}(A)$.
\end{prop}

Conversely, it is readily checked that any group isomorphism $\nu\colon \overline{\K}_1(A)\simeq \mathrm{H}(A)\oplus\K_1(A)$ or equivalently, any retract $r\colon \overline{\K}_1(A)\longrightarrow \mathrm{H}(A)$
 of $i\colon \mathrm{H}(A)\longrightarrow \overline{\K}_1(A)$, induces a unique Nielsen-Thomsen basis.
 
\begin{dfn} Let $A,B$ be $\CatCa$-algebras and $\mathcal{C},\mathcal{D}$ be Nielsen-Thomsen bases of $A,B$. 

For any group morphism $\beta\colon \overline{\K}_1(A)\longrightarrow \overline{\K}_1(B)$, we define the \emph{matrix of $\beta$ in bases $\mathcal{C}\mathcal{D}$}, in symbols $\Mat_{\mathcal{C}\mathcal{D}}(\beta)$, to be the unique 2-by-2 matrix (with group morphisms entries) such that the following diagram commutes
\[
\xymatrix{
\overline{\K}_1(A)\ar[rr]^{\nu_{\mathcal{C}}}\ar[d]_{\beta}&&  \HH(A)\oplus \K_1(A)\ar[d]^{\Mat_{\mathcal{C}\mathcal{D}}(\beta)}\\
\overline{\K}_1(B)\ar[rr]_{\nu_{\mathcal{D}}}&&  \HH(B)\oplus\K_1(B)
}
\]
where $\nu_{\mathcal{C}},\nu_{\mathcal{D}}$ are the unique isomorphisms induced by the respective Nielsen-Thomsen bases. 

We say that $\beta$ is \emph{Nielsen-Thomsen diagonalisable} if, there exist Nielsen-Thomsen bases such that the matrix of $\beta$ is diagonal.
\end{dfn}

\begin{thm}[Rotation map]
Let $\phi\colon A\longrightarrow B$ be a *-homomorphism between unital $\CatCa$-algebras. Let $\mathcal{C}=\{c_k\}_{k\in\K_1(A)},\mathcal{D}=\{d_k\}_{k\in\K_1(B)}$ be Nielsen-Thomsen bases of $A,B$. 

Let $s_A,r_B$ be the respective section and retract induced by the bases. We compute that
\[
\begin{array}{ll} 
r_B\circ \overline{\K}_1(\phi)\circ s_A\colon \K_1(A)\longrightarrow \mathrm{H}(B)\\
\hspace{3,4cm}k\longmapsto \overline{\Delta}_{}
\begin{psmallmatrix} \phi(c_k) & 0 \\
0 & d_{\K_1(\phi)(k)}^*
\end{psmallmatrix}
\end{array}
\] 
which we write $R_{\mathcal{C}\mathcal{D}}(\phi)$ and refer to as the \emph{rotation map of $\phi$ with respect to the bases $\mathcal{C}\mathcal{D}$}.

Furthermore, we have that
\[
\Mat_{\mathcal{C}\mathcal{D}}(\overline{\K}_1(\phi))=
\begin{pmatrix}\mathrm{H}(\phi)  & R_{\mathcal{C}\mathcal{D}}(\phi)\\
0 & \K_1(\phi)
\end{pmatrix}
\]
\end{thm}

\begin{proof}
We know that $\phi$ induces a morphism $(\HH(\phi),\overline{\K}_1(\phi),\K_1(\phi))$ between the Nielsen-Thomsen sequences of $A$ and $B$. (Observe that this triple is a priori not compatible with the chosen splits.) This is equivalent to have that $\overline{\K}_1(\phi)\circ i_A=i_B\circ\HH(\phi)$ and $\K_1(\phi)\circ\pi_A=\pi_B\circ\overline{\K}_1(\phi)$. 

As a result, we obtain that
\[
\left\{\begin{array}{ll}
\HH(\phi)=r_B\circ i_B\circ\HH(\phi)=r_B\circ \overline{\K}_1(\phi)\circ i_A\\
\K_1(\phi)=\K_1(\phi)\circ \pi_A\circ s_A=\pi_B\circ\overline{\K}_1(\phi)\circ s_A
\end{array}
\right.
\]
from which the computations of the entries of $\Mat_{\mathcal{C}\mathcal{D}}(\overline{\K}_1(\phi))$ follow.
\end{proof}

\begin{rmk}
(i) The rotation map quantifies how $\phi$ \textquoteleft affects\textquoteright\ the Nielsen-Thomsen bases chosen.

(ii) It can be shown that rotation maps of $\phi$ with respect to any bases all agree on $\mathrm{Tor}(\K_1(A))$. However, we will be interested in $\A\!\T$-algebras, which have torsion-free $\K_1$-groups, and hence, we do not pursue these ideas here.
\end{rmk}

We obtain an immediate corollary from standard theory of split-extensions of abelian groups.

\begin{cor}\label{prop:commutsquare} Let $A,B$ be $\CatCa$-algebras. Let $\beta\colon \overline{\K}_1(A)\longrightarrow \overline{\K}_1(B)$ be a morphism. The following are equivalent

(i) $\beta$ is Nielsen-Thomsen diagonalisable.

(ii) There exist a section $s_A\colon \K_1(A)\rightarrow \overline{\K}_1(A)$ and a retract $r_B\colon \overline{\K}_1(B)\rightarrow \HH(B)$ such that $r_B\circ\beta\circ s_A =0$ and $\beta\circ i_A\subseteq \im(i_B)$.

If moreover $\beta=\overline{\K}_1(\phi)$, for some *-homomorphism $\phi\colon A\longrightarrow B$, the above are also equivalent to

(iii) The triple $(\HH(\phi),\overline{\K}_1(\phi),\K_1(\phi))$ is split.

(iv) There exist sections $s_A,s_B$ of $\pi_A,\pi_B$ respectively, such that the following square commutes.
\[
\xymatrix{
\overline{\K}_1(A)\ar[r]_{\pi_A}\ar[d]_{\overline{\K}_1(\phi)}& \K_1(A)\ar[d]^{\K_1(\phi)}\ar@/^{-1,2pc}/[l]_{s_A}\\
 \overline{\K}_1(B)\ar[r]^{\pi_b}& \K_1(B)\ar@/_{-1,2pc}/[l]^{s_B}}
\]
\end{cor}

\begin{prg}\textbf{Application to unitary elements of a $\CatCa$-algebra.} Let $A$ be a unital $\CatCa$-algebra. It is well-known that its unitary elements are in bijective correspondence with $\Hom_{\CatCa,1}(C(\T),A)$ by sending $u\in U(A)$ to $\varphi_u\colon C(\T)\longrightarrow A$, where $\varphi_u(\id_\T)=u$. We also recall that $\K_1(C(\T))\simeq \Z$ is generated by $[\id_\T]=1_{\Z}$.

We consider the \emph{canonical Nielsen-Thomsen basis of $C(\T)$} to be the indexed set $\mathcal{C}_0:= \{c_k\}_{k\in\Z}$, where $c_0:=1_A$, and for any $k\geq 1$, we fix
\[
 c_k:=\underset{\text{k-times}}{\begin{psmallmatrix} \id_\T & & \\
&\ddots&\\
&&\id_\T
\end{psmallmatrix}}\quad\text{ and }\quad c_{-k}:=\underset{\text{k-times}}{\begin{psmallmatrix} \id_\T^* & & \\
&\ddots&\\
&&\id_\T^*
\end{psmallmatrix}}
.
\]
\begin{dfn} Let $A$ be a unital $\CatCa$-algebra. Fix a Nielsen-Thomsen basis $\mathcal{C}:=\{c_k\}_{k\in\K_1(A)}$  of $A$. We define the \emph{extended de la Harpe-Skandalis determinant with respect to $\mathcal{C}$} as

\[
\begin{array}{ll}
\overline{\Delta}_{\mathcal{C}}\colon U^\infty(A)\longrightarrow \HH(A) \\
\hspace{1,65cm}u\longmapsto \overline{\Delta}_{}(
\begin{smallmatrix} u & 0 \\
0 & c_{[u]}^*
\end{smallmatrix})
\end{array}.
\]
\end{dfn}

\begin{rmk} 
(i) Observe that $R_{\mathcal{C}_0\mathcal{C}}(\varphi_u)(1_\Z)=\overline{\Delta}_{\mathcal{C}}(u)$. In particular, we compute that 
\[
\Mat_{\mathcal{C}_0\mathcal{C}}(\overline{\K}_1(\varphi_u))=
\begin{pmatrix}\mathrm{H}(\varphi_u)  & 1_\Z \mapsto \overline{\Delta}_{\mathcal{C}}(u)\\
0 & 1_\Z \mapsto[u]
\end{pmatrix}.
\]

(ii) Whenever $u\in U^\infty_0(A)$, we recover the original determinant. That is, 
$\overline{\Delta}_{\mathcal{C}}(u)=\overline{\Delta}_{}(u)$.
\end{rmk}

The above yields an example of a *-homomorphism that is \emph{not} Nielsen-Thomsen diagonalisable.

\begin{exa}[Non-diagonalisable morphism] Let $u\in C([0,1])$ be the unitary element defined by $t\mapsto \exp(2i\pi t)$ and let $\varphi_u\colon C(\T)\rightarrow C([0,1])$ be its associated *-homomorphism. In other words, $\varphi_u$ is the *-homomorphism sending $\id_\T\mapsto u$, defined via functional calculus. 

Let $\mathcal{C}:=\{c_k\}_{k\in \Z}$ be a Nielsen-Thomsen basis for $C(\T)$. By the triviality of $\K_1(C([0,1]))$, there exists a unique Nielsen-Thomsen basis $\mathcal{D}_0$ of $C([0,1])$ given by $\{d_0:=1_{[0,1]}\}$. 

 We aim to prove  that the morphism $R_{\mathcal{C}\mathcal{D}_0}(\varphi_u)\colon \K_1(C(\T))\rightarrow \HH(C([0,1]))$ is non-trivial. First, it is known that $\K_1(C(\T))\simeq \Z$ and that $\HH(C([0,1]))\simeq C_\T([0,1])$. Therefore, it is enough to show that $R_{\mathcal{C}\mathcal{D}_0}(\varphi_u)(1_\Z)\neq 1_{[0,1]}$. 

We observe that $\varphi_u(c_1)=c_1\circ u$, which gives us that $R_{\mathcal{C}\mathcal{D}_0}(\varphi_u)(1_\Z)= \overline{\Delta}(c_1\circ u)$. Then on, it is readily checked that $\overline{\Delta}(c_1\circ u)=1_{[0,1]}$ if and only if $c_1\circ u=1_{[0,1]}$ if and only if $c_1=1_{\T}$, which in turn, implies that $[c_1]_{\K_1}=0_\Z$. Nevertheless, we have that $[c_1]_{\K_1}=1_\Z$ (by definition), from which we deduce that $R_{\mathcal{C}\mathcal{D}_0}(\varphi_u)(1_\Z)\neq 1_{[0,1]}$, for any Nielsen-Thomsen basis $\mathcal{C}$. We conclude that $\overline{\K}_1(\varphi_u)$ is not Nielsen-Thomsen diagonalisable.
\end{exa}
\end{prg}

\subsection{Comparison of *-homomorphisms at the level of the $\overline{\K}_1$-group} The category of abelian groups can be equipped with the discrete topology, which induces the trivial metric $d_{\mathrm{triv}}$ on the set of morphisms. 
While this metric is the finest possible (as it arises from the discrete topology), it is also somewhat \textquoteleft pathological\textquoteright, in the sense that it only distinguishes whether two group morphisms are the same or different. In what follows, we exploit the matrix representations in Nielsen-Thomsen bases to define a more informative, and thus more useful, metric $\mathfrak{d}$ that compares *-homomorphisms at the level of their $\overline{\K}_1$-groups. For the rest of the section, we consider the following setting.

Let $A,B$ be $\CatCa$-algebras. Let $\mathcal{C},\mathcal{D}$ be Nielsen-Thomsen bases of $A,B$ respectively. Let $\phi,\psi\colon A\longrightarrow B$ be *-homomorphisms. Recall that 
\[
\Mat_{\mathcal{C}\mathcal{D}}(\overline{\K}_1(\phi))=
\begin{pmatrix}\mathrm{H}(\phi)  & R_{\mathcal{C}\mathcal{D}}(\phi)\\
0 & \K_1(\phi)
\end{pmatrix} \quad\text{ and }\quad  \Mat_{\mathcal{C}\mathcal{D}}(\overline{\K}_1(\psi))=
\begin{pmatrix}\mathrm{H}(\psi)  & R_{\mathcal{C}\mathcal{D}}(\psi)\\
0 & \K_1(\psi)
\end{pmatrix}.
\]

We aim to construct a metric $\mathfrak{d}$ which compares $\overline{\K}_1(\phi)$ and $\overline{\K}_1(\psi)$, based on entry-wise comparison. A priori, such a comparison depends on the Nielsen-Thomsen bases involved. However, we will see that any Nielsen-Thomsen bases chosen yield equivalent metrics.
  
\begin{prg}\textbf{Comparison of $\HH(\phi)$ and $\HH(\psi)$.} As stated earlier, we freely identify $\Aff T_1(A)\simeq A_{\mathrm{sa}}/A_0$ as complete order unit spaces. Consequently, any element of $\Aff T_1(A)$ is of the form $\widehat{h}:=\ev_h$, where $h\in A_{\mathrm{sa}}$ and, its norm is given by $\Vert \widehat{h}\Vert= \sup_{\tau\in T_1(A)}\{\vert \tau(h)\vert\}$. Furthermore, $\HH(A)$ is equipped with the quotient norm given by $\Vert [\widehat{h}]\Vert:=\inf_{\widehat{p}\in \overline{\im (\underline{\Tr}_A)}}\{\Vert \widehat{h}-\widehat{p}\Vert\}$. As a result, we obtain the metrics 
\[
\left\{\begin{array}{ll}
\vspace{0,2cm}d(\Aff T_1(\phi),\Aff T_1(\psi))= \sup_{h\in A_{\mathrm{sa}}\cap A_1}\{\Vert\widehat{\phi(h)}- \widehat{\psi(h)}\Vert\} \\
d(\HH(\phi),\HH(\psi)):= \sup\limits_{h\in A_{\mathrm{sa}}\cap A_1}\{\Vert[\widehat{\phi(h)}]-[\widehat{\psi(h)}]\Vert\}
\end{array}
\right.
\]
where $A_1$ denotes the closed unit ball of $A$. Observe that
$d(\HH(\phi),\HH(\psi))\leq d(\Aff T_1(\phi),\Aff T_1(\psi))$.
\end{prg}

\begin{prg}\textbf{Comparison of $\K_1(\phi)$ and $\K_1(\psi)$.} We use the trivial metric on this entry. That is, 
\[
d_{\mathrm{triv}}(\K_1(\phi),\K_1(\psi))=\left\{ \begin{array}{ll} 0 \text{ whenever } \K_1(\phi)=\K_1(\psi)\\ \infty \text{ else } \end{array}\right..
\]
\end{prg}

\begin{prg}\textbf{Comparison of the rotation maps of $\phi$ and $\psi$.}
Recall that the rotation maps of $\phi$ and $\psi$ are group morphisms, relying on the Nielsen-Thomsen bases chosen. Explicitly, we have
\[
\begin{array}{ll}
R_{\mathcal{C}\mathcal{D}}(\phi)\colon\K_1(A)\longrightarrow \HH(B) \hspace{2,5cm}\text{ and }\hspace{1cm} R_{\mathcal{C}\mathcal{D}}(\psi)\colon\K_1(A)\longrightarrow \HH(B) \\
\hspace{2cm}k\longmapsto \overline{\Delta}_{}
\begin{psmallmatrix} \phi(c_k) & 0 \\
0 & d_{\K_1(\phi)(k)}^*
\end{psmallmatrix}
\hspace{4,5cm} k\longmapsto \overline{\Delta}_{}
\begin{psmallmatrix} \psi(c_k) & 0 \\
0 & d_{\K_1(\psi)(k)}^*
\end{psmallmatrix}
\end{array}.
\]

Recall that the above codomain is equipped with a metric. As a result, whenever $\K_1(A)$ is countably generated, we are able construct a metric to compare these maps. More concretely, we fix a countable generating set $\mathcal{S}\overset{\iota}{\simeq} \N$ and we define
\[
d_{\mathcal{S},\mathcal{C}\mathcal{D}}(\phi,\psi):=\sum_{k\in \mathcal{S}}\frac{\Vert R_{\mathcal{C}\mathcal{D}}(\phi)(k) - R_{\mathcal{C}\mathcal{D}}(\psi)(k)\Vert}{2^{\iota(k)}(1+\Vert R_{\mathcal{C}\mathcal{D}}(\phi)(k) - R_{\mathcal{C}\mathcal{D}}(\psi)(k)\Vert)
}.
\]

While it is true that any countable generating set induces the \emph{point-wise convergence topology}, it need not be true that they all generate equivalent metrics. A similar statement is true for different pairs of Nielsen-Thomsen bases. It worth mentioning that, whenever $\K_1(A)$ is finitely generated, any finite generating set $\mathcal{S}$ induces a metric equivalent to
\[
d_{\mathcal{S},\mathcal{C}\mathcal{D}}(\phi,\psi):=\max_{k\in \mathcal{S}}\Vert R_{\mathcal{C}\mathcal{D}}(\phi)(k) - R_{\mathcal{C}\mathcal{D}}(\psi)(k)\Vert.
\]

To ease notations, we may write $d_{\mathcal{C}\mathcal{D}}$, whenever $\K_1(A)$ is finitely generated.
\end{prg}

\begin{lma}\label{lma:key}
Let $A,B$ be $\CatCa$-algebras. Let $\mathcal{C}:=\{c_k\}_{k\in \K_1(A)},\mathcal{D}$ be Nielsen-Thomsen bases of $A,B$ respectively. Let $\phi,\psi\colon A\longrightarrow B$ be *-homomorphisms. Assume moreover that $K_1(\phi)=\K_1(\psi)$. 

For any $k\in\K_1(A)$, we compute
 \[
(R_{\mathcal{C}\mathcal{D}}(\phi)-R_{\mathcal{C}\mathcal{D}}(\psi))(k)=\overline{\Delta}_{}
\begin{psmallmatrix} \phi(c_k) & 0 \\
0 & \psi(c_k)^*
\end{psmallmatrix}.
\] 
In particular, the above morphism only depends on the basis chosen for $A$.
\end{lma}

\begin{proof}
Since we have $d_{\K_1(\phi)(k)}=d_{\K_1(\psi)(k)}$, the result follows from the fact that $\overline{\Delta}(u)+\overline{\Delta}(u^*)=0$.
\end{proof}

Gathering all the above, we are now ready to construct a (non-trivial) metric to compare the *-homomorphisms at the level of the Hausdorffized algebraic unitary group. Most importantly, we will see that the topology induced by this metric does not depend on the Nielsen-Thomsen bases.

\begin{dfn}\label{dfn:adequatemetric} Let $\phi,\psi\colon A\longrightarrow B$ be *-homomorphisms between $\CatCa$-algebras $A,B$. Let $\mathcal{C},\mathcal{D}$ be Nielsen-Thomsen bases of $A,B$ and let $\mathcal{S}$ be a generating set of $\K_1(A)$. We set
\[
\mathfrak{d}_{\mathcal{S},\mathcal{C}\mathcal{D}}(\overline{\K}_1(\phi),\overline{\K}_1(\psi)):= d(\HH(\phi),\HH(\psi))+d_{\mathcal{S},\mathcal{C}\mathcal{D}}(\phi,\psi)+d_{\mathrm{triv}}(\K_1(\phi),\K_1(\psi)).
\]

This is a well-defined metric on the set $\Hom(\overline{\K}_1(A),\overline{\K}_1(B))$. To ease notations, we may write $\mathfrak{d}_{\mathcal{C}\mathcal{D}}$, whenever $\K_1(A)$ is finitely generated.
\end{dfn}

\begin{thm}
Let $\mathcal{C},\mathcal{C}'$ be Nielsen-Thomsen bases of $A$ and $\mathcal{D},\mathcal{D}'$ be Nielsen-Thomsen bases of $B$. Assume that $\K_1(A)$ is finitely generated. 
For any *-homomorphisms $\phi,\psi\colon A\longrightarrow B$, we have
\[
\mathfrak{d}_{\mathcal{C}\mathcal{D}}(\overline{\K}_1(\phi),\overline{\K}_1(\psi))\leq 2\mathfrak{d}_{\mathcal{C}'\mathcal{D}'}(\overline{\K}_1(\phi),\overline{\K}_1(\psi)).
\]

In particular, $\mathfrak{d}_{\mathcal{C}\mathcal{D}}$ and $\mathfrak{d}_{\mathcal{C}'\mathcal{D}'}$ are equivalent metrics.
\end{thm}

\begin{proof}
If $\phi$ and $\psi$ disagree on the $\K_1$-group, both values are infinite and we are done. Otherwise, $d_{\mathrm{triv}}(\K_1(\phi),\K_1(\psi))=0$. In such case, fix $k\in \K_1(A)$. By \autoref{lma:key}, we compute that 
\begin{align*}
(R_{\mathcal{C}\mathcal{D}}(\phi)-R_{\mathcal{C}\mathcal{D}}(\psi))(k)-(R_{\mathcal{C}'\mathcal{D}'}(\phi)-R_{\mathcal{C}'\mathcal{D}'}(\psi))(k)&=\overline{\Delta}_{}
\begin{psmallmatrix} \phi(c_k) & 0 \\
0 & \psi(c_k)^*
\end{psmallmatrix} -
\overline{\Delta}_{}
\begin{psmallmatrix} \phi(c'_k) & 0 \\
0 & \psi(c'_k)^*
\end{psmallmatrix}\\
&=\overline{\Delta}_{}
\begin{psmallmatrix} \phi(c_k) & 0 \\
0 & \phi(c'_k)^*
\end{psmallmatrix} -
\overline{\Delta}_{}
\begin{psmallmatrix} \psi(c_k) & 0 \\
0 & \psi(c'_k)^*
\end{psmallmatrix}\\
&=\HH(\phi)(z_k)-\HH(\psi)(z_k)
\end{align*}
where $z_k:=\overline{\Delta}_{}
\begin{psmallmatrix} c_k & 0 \\
0 & (c'_k)^*
\end{psmallmatrix}\in\HH(A)$. 

Since $\K_1(A)$ is finitely generated, say by $\mathcal{S}$, we easily deduce that $d_{\mathcal{C}\mathcal{D}}(\phi,\psi)\leq (\HH(\phi)-\HH(\psi))(z_l) +(R_{\mathcal{C}'\mathcal{D}'}(\phi)-R_{\mathcal{C}'\mathcal{D}'}(\psi))(l)$, for some $l\in \mathcal{S}$. This yields the inequality \[d_{\mathcal{C}\mathcal{D}}(\phi,\psi)\leq d(\HH(\psi),\HH(\phi))+d_{\mathcal{C}'\mathcal{D}'}(\phi,\psi).\] The result now follows by adding $d_{\mathcal{C}'\mathcal{D}'}(\phi,\psi)$ on the right side of the inequality, and $d(\HH(\psi),\HH(\phi))$ on both sides.
\end{proof}

Even though the Nielsen-Thomsen bases involved will often be unambiguous, we may abusively denote any such (equivalent) metrics by $\mathfrak{d}$. In the last section of the manuscript, we exhibit examples satisfying that $0<\mathfrak{d}(\overline{\K}_1(\phi),\overline{\K}_1(\psi))<\infty$, showcasing the usefulness of $\mathfrak{d}$.

\begin{prg}\textbf{Application to the Hausdorffized unitary Cuntz semigroup.}\label{prg:metricsFinal} Applying the comparison methods obtained in \autoref{subsection:2B} with the above metrics, yields explicit ways to compare *-homomorphisms at the level of unitary Cuntz semigroups constructed in \autoref{dfn:unitaryCu}. 

Let $A,B$ be $\CatCa$-algebras and let $\phi,\psi\colon A\rightarrow B$ be *-homomorphisms. Assume moreover that $\Cu(A)\simeq \Lsc(X,\overline{\N})$ for some locally compact Hausdorff space $X$. From \autoref{dfn:metricK}, we obtain the following metrics.
\[
\left\{
\begin{array}{ll}
\vspace{0,2cm}d^*_{\Cu}(\Cu_{\K_1}(\phi),\Cu_{\K_1}(\psi)):=\inf_{} \Bigl\{r>\epsilon_0\mid \forall U\in \mathcal{O}(X),  \, \mathcal{F}(\mymathbb{1}_U,\mymathbb{1}_{U_r}) \text{ commutes}   \Bigr\}\\
\mathfrak{d}^*_{\Cu}(\Cu_{\overline{\K}_1}(\phi),\Cu_{\overline{\K}_1}(\psi)):=\inf_{} \Bigl\{r>\epsilon_0\mid  U\in \mathcal{O}(X), \, \Vert  \mathcal{F}(\mymathbb{1}_U,\mymathbb{1}_{U_r})\Vert_\mathfrak{d}\leq 4r \Bigr\}
\end{array}
\right.
\]
where $\epsilon_0:=d_{\Cu}(\alpha_0,\beta_0)$ and the fiber diagrams respectively involve $\K_1$-groups and $\overline{\K}_1$-groups. 

For notational purposes, we write $d^*_{\Cu}$ to mean $d^*_{\Cu,d_\mathrm{triv}}$. Similarly, we write $\mathfrak{d}^*_{\Cu}$ to mean $d^*_{\Cu,\mathfrak{d}}$. 
An analogous approach, using the finite-set comparison, can be done to remove the assumption on the Cuntz semigroup of $A$. More broadly, we can use the $\Cu$-metrics and the theory of comparison developed in \cite[Section 5]{CV24} to construct similar metrics. We do not pursue such developments here.

  We end this part by summarizing and relating all the comparisons of *-homomorphisms at the level of the Cuntz semigroup and its refined versions that we have considered so far. 
  
\begin{prop} Let $X$ be a locally compact Hausdorff space. Let $B$ be a $\CatCa$-algebra and let $\phi,\psi\colon C(X)\longrightarrow B$ be *-homomorphisms. We have
\begin{align*}
d_{\Cu}(\Cu(\phi),\Cu(\psi)) &\leq d^*_{\Cu}(\Cu_{\K_1}(\phi),\Cu_{\K_1}(\psi)) \\
&\leq \mathfrak{d}^*_{\Cu}(\Cu_{\overline{\K}_1}(\phi),\Cu_{\overline{\K}_1}(\psi))\\& \leq d^*_{\Cu}(\Cu_{\overline{\K}_1}(\phi),\Cu_{\overline{\K}_1}(\psi))\\
&\leq d_U(\phi,\psi).
\end{align*}
Here, $d_U(\phi,\psi):=\sup\limits_{F\subset{\rm Lip^1(X)}}\inf\limits_{w\in\mathcal{U}(A)}\sup\limits_{f\in F} \Vert w\phi(f)w^*-\psi(f)\Vert$ is the unitary distance between $\phi$ and $\psi$. See \cite[Definition 3.1]{JSV21}. 
\end{prop} 

\begin{proof} The result follows from the observations that 
$
d_{\mathrm{triv}}(\K_1(\phi),\K_1(\psi))\leq \mathfrak{d}(\overline{\K}_1(\phi),\overline{\K}_1(\psi))\leq d_{\mathrm{triv}}(\overline{\K}_1(\phi),\overline{\K}_1(\psi))
$
and that $d_{\mathrm{triv}}(\K_1(\phi),\K_1(\psi))=0$ if and only if $\mathfrak{d}(\overline{\K}_1(\phi),\overline{\K}_1(\psi)) <\infty$. 
\end{proof}
\end{prg}

\begin{prg}\textbf{Recovering tracial states from the Cuntz semigroup.} The relationship between the tracial states and the Cuntz semigroup of a $\CatCa$-algebra has a long history. See e.g. \cite{BPT08, C78, CP79, ERS11}. Based on \cite[Theorem 4.4]{ERS11}, it would seem that the Cuntz semigroup \textquoteleft contains\textquoteright\ the information of tracial states. However, we ought to make the latter statement more precise. As a first step, let us state the following theorem.
 
For the reader's convenience, we recall that a \emph{functional on a Cu-semigroup} is a monoid morphism  preserving suprema of increasing sequences, whose codomain is $\overline{\N}$. Also, any tracial state on a $\CatCa$-algebra $A$ uniquely corresponds to a functional on the Cuntz semigroup of $A$. See \cite{ERS11} for the explicit correspondance and more on the matter.

\begin{thm}\label{thm:Cufinernose}
Let $\phi,\psi\colon A\longrightarrow B$ be *-homomorphisms between $\CatCa$-algebras. Assume that $\phi$ and $\psi$ agree at the level of the Cuntz semigroup. 

Then $\Aff T_1(\phi)=\Aff T_1(\psi)$. A fortiori, we have $\HH(\phi)=\HH(\psi)$.
\end{thm}

\begin{proof}
Let $\tau\in T_1(A)$. By \cite[Theorem 4.4]{ERS11}, $\tau$ induces a unique functional $\lambda_\tau$ on $\Cu(A)$. For any $a\in A_+$, we compute the following.
\begin{align*}
\tau(\phi(a))&=\int_0^{\infty} \lambda_\tau([(\phi(a)-t)_+])dt=\int_0^{\infty} \lambda_\tau([\phi((a-t)_+)])dt\\
&=\int_0^{\infty} \lambda_\tau\circ\Cu(\phi)([(a-t)_+])dt=\int_0^{\infty} \lambda_\tau\circ\Cu(\psi)([(a-t)_+])dt\\
&=\tau(\psi(a)).
\end{align*}

We have used the formula of \cite[Proposition 4.2]{ERS11} at the first and last equality. The second equality follows from the fact that *-homomorphisms preserve continuous functional calculus, while the fourth equality follows from the hypothesis at hand.
\end{proof}

In the aim of simplifying the comparison of *-homomorphisms at the level of the Hausdorffized unitary Cuntz semigroup, given by the metric $\mathfrak{d}^*_{\Cu}$, we wonder whether an approximate version of the above theorem remains true. Concretely, we ask the following.

\begin{qst} Retain the setting of the latter theorem. Is there a $\Cu$-metric $d_{\Cu}$ (together with a positive constant $C>0$) satisfying that $d(\Aff T_1(\phi),\Aff T_1(\psi))\leq C d_{\Cu}(\Cu(\phi),\Cu(\psi))$?

In a weaker fashion, can we obtain that $d(\HH(\phi),\HH(\psi))\leq C d_{\Cu}(\Cu(\phi),\Cu(\psi))$?
\end{qst}

In the particular case where $A=C([0,1])$ and $B$ is a $\CatCa$-algebra of stable rank one, both of the above have an affirmative answer, following from the classification obtained in \cite{R12}. It is worth mentioning that a theory of $\Cu$-metrics has been developed in \cite[Section 5]{CV24} and that this line of investigation is part of an on-going research. See \cite{C26}.
\end{prg}

\section{Distinguishing \texorpdfstring{$\CatCa$}{C*}-algebras and *-homomorphisms}\label{sec:Sec5}
In this part, we apply our comparison methods, previously set, to distinguish concrete examples of $\CatCa$-algebras and *-homomorphisms that neither the original Elliott invariant nor the (unitary) Cuntz semigroup are able to distinguish. We remark that these examples deals with $\A\!\T$-algebras and their morphisms. They highlight the neatness and applicability of our methods to classification problems.

\subsection{The Gong-Jiang-Li example revisited}
In \cite{GJL20}, the authors have built (non-simple) $\A\!\T$-algebras $A$ and $B$ that are distinguished by a refined version of the Elliott invariant. More particularly, these $\CatCa$-algebras differ from an internal property, called \emph{uniformly varied determinants} of connecting maps, which in turn implies a system of splitting between specific corner algebras. In the original manuscript, both the invariant and the distinguishing arguments are quite long to unravel. Also, it is stated without any proof that these $\CatCa$-algebras cannot be distinguished by their Cuntz semigroup. 

We aim to distinguish these $\CatCa$-algebras via the Hausdorffized unitary Cuntz semigroup and our methods, providing a more conceptual and streamlined proof of the non-isomorphism between $A$ and $B$. Let us start by recalling their constructions.

$\bullet$\,\,\textbf{Construction of the blocks.} 

Let $(p_n)_{n\geq 1}$ denote the sequence of all the prime numbers in increasing order. 

Let $(k_n)_{n\geq 1}$ be a strictly increasing sequence of natural numbers such that $k_1\geq 2$.\\
The building blocks of the inductive system for $A$ and $B$ are defined by
\begin{align*}
&A_1 = B_1 = C(\T)\\
&A_2 = B_2 = M_{p_1^{k_1}}(C[0,1]) \oplus M_{p_1^{k_1}}(C(\T))\\
&A_3 = B_3 = M_{p_1^{k_1}p_1^{k_2}}(C[0,1]) \oplus M_{p_1^{k_1}p_2^{k_2}}(C[0,1])\oplus M_{p_1^{k_1}p_2^{k_2}}(C(\T))\\
&A_4 = B_4 = M_{p_1^{k_1}p_1^{k_2}p_1^{k_3}}(C[0,1]) \oplus M_{p_1^{k_1}p_2^{k_2}p_2^{k_3}}(C[0,1]) \oplus M_{p_1^{k_1}p_2^{k_2}p_3^{k_3}}(C[0,1])\oplus M_{p_1^{k_1}p_2^{k_2}p_3^{k_3}}(C(\T))\\
&\quad\vdots
\end{align*}

Let $n\in\N$. Write $[n,i]:=\prod_{j=1}^{i}p_j^{k_j}\prod_{j=i+1}^{n-1}p_i^{k_j}$, for any $1\leq i\leq n-1$ and $[n,n]:=[n,n-1]$. (Notice that $[n+1,i]=p_i^{k_n}[n,i]$ for any $1\leq i\leq n-1$.) 

We remark that these notations yield the following induction formula.
\[
    		A_n=B_n=\bigoplus\limits_{i=1}^{n-1}M_{[n,i]}(C[0,1])\oplus M_{[n,n]}(C(\T)).
\]

$\bullet$\,\,\textbf{Construction of the connecting maps.} 

Let $(t_n)_{n\geq 1}$ be a countable dense subset of $[0,1]$.

Let $(z_n)_{n\geq 1}$ be a countable dense subset of $\T$. \\
Let $n \in\N$. The connecting maps $\phi_{nn+1}: A_n\longrightarrow A_{n+1}$ and $\psi_{nn+1}: B_n\longrightarrow B_{n+1}$ of the inductive systems are defined also inductively, via partial *-homomorphisms as follows.

$\bullet$ \emph{The $n-1$ first partial *-homomorphisms for $\phi_{nn+1}$ and $\psi_{nn+1}$}. Let $1\leq i\leq n-1$.
\[
\begin{array}{ll}
\phi_{nn+1}^i=\psi_{nn+1}^i\colon M_{[n,i]}(C[0,1])\longrightarrow M_{[n+1,i]}(C[0,1])\\
\hspace{4,8cm} f\longmapsto 
\begin{psmallmatrix}
f\\
&\ddots\\
&&f\\
&&&f(t_n)
\end{psmallmatrix}
\end{array}
\]

$\bullet$ \emph{The $n$-th partial $^*$-homomorphism for $\phi_{nn+1}$ and $\psi_{nn+1}$}. Let $\exp\colon [0,1]\rightarrow \T$ be the exponential map sending $t\mapsto e^{2i\pi t}$.
\[
\begin{array}{ll}
\phi_{nn+1}^n\colon M_{[n,n]}(C(\T))\longrightarrow M_{[n+1,n]}(C[0,1]) \\
\hspace{3cm} f\longmapsto 
\begin{psmallmatrix}
f\circ\, \exp \\
&f\circ (\exp^{-1})\\
&&f\circ\,\exp(1/r_n)\\
&&&\ddots \\
&&&&f\circ\,\exp(r_n-1/r_n)
\end{psmallmatrix}
\end{array}
\]
where $r_n:=p_n^{k_n}-1$. Now let $l_n:=4^n[n,n+1]$.
\[
\begin{array}{ll}
\psi_{nn+1}^n\colon M_{[n,n]}(C(\T))\longrightarrow M_{[n+1,n]}(C[0,1]) \\
\hspace{3cm} f\longmapsto 
\begin{psmallmatrix}
f\circ \,\exp^{l_n} \\
&f\circ\, \exp(0)\\
&&f\circ\,\exp(1/r_n)\\
&&&\ddots \\
&&&&f\circ\,\exp(r_n-1/r_n)
\end{psmallmatrix}
\end{array}
\]

$\bullet$ \emph{The $(n+1)$-th partial $^*$-homomorphism for $\phi_{nn+1}$ and $\psi_{nn+1}$}. 
\[
\begin{array}{ll}
\phi_{nn+1}^{n+1}=\psi_{nn+1}^{n+1}\colon M_{[n,n]}(C(\T))\longrightarrow M_{[n+1,n+1]}(C(\T)) \\
\hspace{4,4cm} f\longmapsto 
\begin{psmallmatrix}
f \\
&f(z_n)\\
&&\ddots \\
&&&f(z_n)
\end{psmallmatrix}
\end{array}
\]
We finally define
\[
A:=\lim\limits_{\underset{n}{\longrightarrow}}(A_n,\phi_{nn+1})\quad\quad\text{ and } \quad\quad
B:=\lim\limits_{\underset{n}{\longrightarrow}}(B_n,\psi_{nn+1})
\]
where
$
\phi_{nn+1}=(\phi_{nn+1}^1,..., \phi_{nn+1}^{n-1},(\phi_{nn+1}^{n}, \phi_{nn+1}^{n+1}))$ and $\psi_{nn+1}=(\psi_{nn+1}^1,..., \psi_{nn+1}^{n-1},(\psi_{nn+1}^{n}, \psi_{nn+1}^{n+1}))
$.

\begin{prop}
Both $A$ and $B$ are separable unital stable rank one $\CatCa$-algebras.
\end{prop}

\begin{proof}
Since all $\CatCa$-algebras of the inductive systems are separable and unital, together with the fact that all morphisms are also unital, we easily obtain that $A$ and $B$ are unital separable $\CatCa$-algebras. In addition, the stable rank one property is preserved by inductive limits and any interval or circle algebra has stable rank one. 
\end{proof}

Next, we dive into the lattice of ideals of both $A$ and $B$. More particularly, we exhibit the set of simple ideals of these $\CatCa$-algebras together with another set of peculiar ideals that will be of use to distinguish $A$ and $B$ by means of the Hausdorffized unitary Cuntz semigroup. 

\begin{lma} 
\label{lma:ideals}
Let $j\geq 1$. We consider the following inductive systems
\[
\left\{
    \begin{array}{ll}
    		\mathfrak{s}_{j}:=\lim\limits_{\underset{n> j}\longrightarrow}(M_{[n,j]}(C[0,1])),{\phi^j_{nm}}),\quad\quad\mathfrak{p}_{j}:=\lim\limits_{\underset{n\geq j}\longrightarrow}(I_{n,j},{\phi_{nm}}_{\mid I_{n,j}})\quad\text{ and } \quad\mathfrak{a}_{j}:=\lim\limits_{\underset{n\geq j}\longrightarrow}(I^c_{n,j},{\phi_{nm}}_{|I^c_{n,j}})\\
    		\mathfrak{t}_{j}:=\lim\limits_{\underset{n> j}\longrightarrow}(M_{[n,j]}(C[0,1])),{\psi^j_{nm}}),\quad\quad\mathfrak{q}_{j}:=\lim\limits_{\underset{n\geq j}\longrightarrow}(I_{n,j},{\psi_{nm}}_{\mid I_{n,j}})\quad\text{ and } \quad\mathfrak{b}_{j}:=\lim\limits_{\underset{n\geq j}\longrightarrow}(I^c_{n,j},{\psi_{nm}}_{|I^c_{n,j}})
    \end{array}
\right.
\]
where  $I_{n,j}=\overset{n-1}{\underset{i\geq j}\bigoplus}M_{[n,i]}(C[0,1])\,\oplus\, M_{[n,n]}(C(\T))$ and $I^c_{n,j}=\overset{j-1}{\underset{i=1}\bigoplus}M_{[n,i]}(C[0,1])$ are complementary of one another in $A_n$ and $B_n$.

(i) The sets of simple ideals of $A$ and $B$ are respectively $\{\mathfrak{s}_{j}\}_{j\geq 1}$ and $\{\mathfrak{t}_{j}\}_{j\geq 1}$. 

(ii) We have $\mathfrak{p}_{j}=\overset{j-1}{\underset{i=1}{\oplus}}\mathfrak{s}_i$ and $\mathfrak{q}_{j}=\overset{j-1}{\underset{i=1}{\oplus}}\mathfrak{t}_i$.

(iii)  Both $\mathfrak{a}_{j}$ and $\mathfrak{b}_{j}$ are ideals of $A$ and $B$ respectively and 
\[
    \left\{
    \begin{array}{ll}
    		A=\mathfrak{a}_j\oplus \mathfrak{p}_{j}\\
    		B=\mathfrak{b}_j\oplus \mathfrak{q}_{j}  
    		\end{array}
\right..
\]
\end{lma}

\begin{proof} This is done similarly as in the proof of \cite[Theorem 3.2]{C23a}.
\end{proof}

Not only has it been proved in \cite{GJL20} that $A$ and $B$ agree at the level of the (original) Elliott invariant, but also it has been stated that they agree at the level of the Cuntz semigroup. Even though the latter statement is true, we would like to offer a rigorous proof, via an approximate intertwining argument. 

Approximate intertwinings in the category $\Cu$ have been firstly developed in \cite{C22}, for specific $\Cu$-semigroups. Subsequently, this technique has been generalized in \cite{CV24}, to the entire category $\Cu$, and its subcategory $\Cu^*$. Although the theorems have been stated in greater generality, using the finite-set comparison, we restate them in our specific context, as follows.

\begin{thm}[{\cite[Theorem 3.16]{C22} - \cite[Theorem 3.17]{CV24}}]
Let $(S_i,\sigma_{ij}),(T_i,\tau_{ij})$ be inductive sequences in the category $\Cu^*$. 

Assume that there are $\Cu^*$-morphisms $\alpha_i\colon S_i\rightarrow T_i$ and $\beta_i\colon T_i\rightarrow S_{i+1}$ together with a metric $d^*$ such that 
\[d^*(\beta_i\circ\alpha_i,\sigma_{ii+1})<1/2^i\quad\text{and}\quad d^*(\alpha_{i+1}\circ\beta_i,\tau_{ii+1})<1/2^i.\]

Then there exists a $\Cu^*$-isomorphism between $\underset{\rightarrow}{\lim}(S_i,\sigma_{ij})\simeq \underset{\rightarrow}{\lim}(T_i,\tau_{ij})$.
\end{thm}

We may abusively write $a\leq x,y$ to mean that $a\leq x$ and $a\leq y$. Similarly, we may write $x,y\leq a$.

\begin{lma}\label{lma:computeCu} Let $n\geq 1$. We compute that 
\[
d_{\Cu}(\Cu(\phi_{nn+1}^n),\Cu(\psi_{nn+1}^n))=d^*_{\Cu}(\Cu_{\K_1}(\phi_{nn+1}^n),\Cu_{\K_1}(\psi_{nn+1}^n))\leq 1/r_n.
\]
\end{lma}

\begin{proof}
First, let us compute the $\Cu$-distance.  Let $U\subseteq \T$ be an open set. Let $k_U\in\N$ be the number of elements of the set $\{\exp(k/r_n)\mid 1\leq k\leq r_n-1\}$ belonging to $U$. We compute that
\[
k_U\leq\Cu(\phi_{nn+1}^n)(\mymathbb{1}_U)(t),\Cu(\psi_{nn+1}^n)(\mymathbb{1}_U)(t)\leq k_U+2.
\]
Further, we either have that $U_{1/r_n}=\T$ or else, $k_{U_{1/r_n}}= k_U+2$. In both cases, we obtain the following inequalities, for any open set $U\subseteq \T$.
\[
\Cu(\phi_{nn+1}^n)(\mymathbb{1}_{U})\leq \Cu(\psi_{nn+1}^n)(\mymathbb{1}_{U_{1/r_n}})\quad\text{and}\quad \Cu(\psi_{nn+1}^n)(\mymathbb{1}_{U})\leq \Cu(\phi_{nn+1}^n)(\mymathbb{1}_{U_{1/r_n}}).
\]
Secondly, let us compute the $\Cu^*$-distance.  Let $U\subseteq \T$ be an open set. We easily see that $k_{U_{1/r_n}}>0$. Therefore, we compute that for any $U\subseteq \T$
\[
\Cu(\phi)(\mymathbb{1}_{U_{1/r_n}}),\Cu(\psi)(\mymathbb{1}_{U_{1/r_n}})\geq \mymathbb{1}_{[0,1]}.
\]

Let us write $x_A=\Cu(\phi)(\mymathbb{1}_U),y_A=\Cu(\phi)(\mymathbb{1}_{U_{1/r_n}})$, and similarly we write $x_B=\Cu(\psi)(\mymathbb{1}_U),y_B=\Cu(\phi)(\mymathbb{1}_{U_{1/r_n}})$. We know that $\K_1(I_{y_A})\simeq\K_1(I_{y_B})\simeq \{0\}$. As a consequence, the fiber diagram of $(\Cu(\phi),\Cu(\psi))$ at coordinates $(\mymathbb{1}_U,\mymathbb{1}_{U_{r_n}})$ trivially commutes, which ends the proof.
\end{proof}

\begin{thm}\label{thm:Cuiso} We have the following isomorphisms.

(i) $\Cu(A)\simeq \Cu(B)$.

(ii) $\Cu_{\K_1}(A)\simeq \Cu_{\K_1}(B)$. In particular, $A$ and $B$ agree at the level of $\K_0$ and $\K_1$.\\
Furthermore, any (scaled) $\Cu$-isomorphism $\alpha_0\colon \Cu(A)\simeq \Cu(B)$ maps $\mathfrak{p}_j\mapsto \mathfrak{q}_j$, for all $j\geq 1$.
\end{thm}

\begin{proof}
To prove (i) and (ii), we note that $A_n=B_n$ and that $\phi_{nn+1}^i=\psi_{nn+1}^i$, for all $n\in\N$ and all $i\leq n+1$ such that $i\neq n$. Therefore we can apply the approximate intertwining theorem with identity maps between the sequences involved and the computation obtained in the previous lemma.

Now, consider any scaled $\Cu$-isomorphism $\alpha_0\colon \Cu(A)\simeq \Cu(B)$. 
Observe that $\Cu(I)\simeq \Cu(I_{\alpha_0})$. In particular, we have an isomorphism $\Cu(I)_c\simeq \Cu(I_{\alpha_0})_c$ between their monoids of compact elements. In the stable rank one context, this implies that $\K_0(I)\simeq \K_0(I_{\alpha_0})$. Furthermore, simple ideals of $A$ are mapped to simple ideals of $B$. Additionally, it is readily computed that, for all $j\geq 1$
\[
\K_0(\mathfrak{s}_j)\simeq \K_0(\mathfrak{t}_j)\simeq \Z[\frac{1}{p_j}].
\]
Therefore, $\theta_{\alpha_0}$ maps $\mathfrak{s}_j\mapsto \mathfrak{t}_j$ and $\alpha_0$ maps $[1_{\mathfrak{s}_j}]\mapsto [1_{\mathfrak{t}_j}]$. 

Finally, we observe that $[1_A]=\sum_{k=1}^{j-1}[1_{\mathfrak{s}_k}]+[1_{\mathfrak{p}_j}]$, for any $j\geq 1$. Similarly for $B$. Again, in the stable rank one context, it is well-known that $\Cu(A)$ and $\Cu(B)$ have cancellation of compact elements. (See \cite[Theorem 4.3]{RW10}.) As a result, we deduce that $\alpha_0([1_{\mathfrak{p}_j}])=[1_{\mathfrak{q}_j}]$, for any $j\geq 1$.
\end{proof}

\begin{thm}\label{thm:diago}
Let $j\geq 1$. Consider the canonical embeddings $\iota^A_j\colon \mathfrak{p}_j\lhook\joinrel\longrightarrow A$ and $\iota^B_j\colon \mathfrak{q}_j\lhook\joinrel\longrightarrow B$. 

(i) $\{\overline{\K}_1(\iota^A_j)\}_{j\geq 1}$ are simultaneously Nielsen-Thomsen diagonalisable.

(ii) $\{\overline{\K}_1(\iota^B_j)\}_{j\geq 1}$ are not simultaneously Nielsen-Thomsen diagonalisable.
\end{thm}

\begin{proof}
(i) Let $j\geq 1$. We define $u_j:=\diag(\id_\T,\overset{[j,j]-1\text{-times}}{\overbrace{1,\dots,1}})\in M_{[j,j]}(C(\T))$, and we denote $\mathcal{C}_j $ to be the Nielsen-Thomsen basis of $\mathfrak{p}_j$ induced by the section $s_j\colon 1_\Z\mapsto \phi_{j\infty}(u_j)$.
Let us check that $\overline{\K}_1(\iota^A_j)$ is diagonalisable in the bases $\mathcal{C}_j,\mathcal{C}_1$. 
By \autoref{prop:commutsquare}, it is enough to have that \[\overline{\Delta}_{}
\begin{psmallmatrix}\phi_{j\infty}(u_j) & 0 \\
0 & \phi^*_{1\infty}(u_1)
\end{psmallmatrix}=0 \text{ in } \HH(A).\] This is readily deduced from the partial *-homomorphisms defining the connecting maps of $A$.

(ii) Let $j\geq 1$. Observe that $B=\mathfrak{q}_1\supseteq\mathfrak{q}_j\supseteq \mathfrak{q}_{j+1}$, and that $0=\mathfrak{b}_1\subseteq \mathfrak{b}_j\subseteq \mathfrak{b}_{j+1}$. Combined with \autoref{lma:ideals} (ii), we deduce that \[\theta^i_{j}\colon\mathfrak{q}_{j}\underset{\iota^B_{j}}{\lhook\joinrel\longrightarrow B}\overset{\pi_{i}}{\relbar\joinrel\twoheadrightarrow}\mathfrak{b}_{i} \text{ is the zero morphism, for any } 1\leq i\leq j.\] 
Now, assume that $\overline{\K}_1(\iota^B_j)$ is diagonalisable. By \autoref{prop:commutsquare}, there exist sections $s_j,s_1$ of $\overline{\K}_1(\mathfrak{q}_j)\twoheadrightarrow \K_1(\mathfrak{q}_j), \overline{\K}_1(B)\twoheadrightarrow \K_1(B)$ respectively, such that, for any $1\leq i\leq j$, the following diagram commutes.
\[
\xymatrix@!R@!C@R=26pt@C=30pt{
\overline{\K}_1(\mathfrak{q}_j)\ar@/^{-4pc}/[dd]_{0}\ar[r]^{}\ar[d]_{\overline{\K}_1(\iota^B_j)}& \K_1(\mathfrak{q}_j)\simeq \Z\ar[d]^{\K_1(\iota^B_j))=\id_\Z}\ar@/^{-1,2pc}/[l]_{s_j}\\
 \overline{\K}_1(B)\ar[d]_{\overline{\K}_1(\pi_i)}\ar[r]_{}& \K_1(B)\ar@/_{-1,2pc}/[l]^{s_1}\simeq \Z\\
 \overline{\K}_1(\mathfrak{b}_i)
}
\]

Assume that  $\{\overline{\K}_1(\iota^B_j)\}_{j\geq 1}$ are simultaneously diagonalisable in the bases $\{\mathcal{C}_j\}_{j\geq 1}$. The above implies that $\overline{\K}_1(\pi_i)\circ s_1$ is trivial, for all $i\geq 1$, where $s_1$ is the section induced by $\mathcal{C}_1$. However, it has been shown in \cite[pp. 73-74]{GJL20}, that for any section $s_1$ and any representative $u\in B$ of $s_1(1_\Z)$, there exists $m\in\N$ large enough, such that  $\Vert \overline{\K}_1(\pi_m)\circ s_1(1_\Z)\Vert\geq 3$, leading to a contradiction.
\end{proof}

\begin{cor}
There is no $\Cu^*$-isomorphism between $\Cu_{\overline{\K}_1}(A)$ and $\Cu_{\overline{\K}_1}(B)$. 
\end{cor}

\begin{proof} 
Assume that there exists a $\Cu^*$-isomorphism $\alpha\colon\Cu_{\overline{\K}_1}(A)\longrightarrow \Cu_{\overline{\K}_1}(A)$. We know that $\alpha$ induces a $\Cu$-isomorphism $\alpha_0\colon \Cu(A)\longrightarrow \Cu(B)$, which in turn induces a lattice isomorphism $\theta_{\alpha_0}\colon\Lat(A)\simeq\Lat(B)$. Furthermore, by \autoref{thm:Cuiso}, we know that $\theta_{\alpha_0}(\mathfrak{p}_j)= \mathfrak{q}_j$, for any $j\geq 1$.

Now, by the properties of the $\Cu_\K$-constructions recalled in \autoref{thm:exactCUK} (ii)/(iii) -see also \cite[Theorem 4.5]{C23b}-, we know that the following diagram is commutative with exact rows
\[
\xymatrix{
0\ar[r]^{} & \Cu(\mathfrak{p}_j)\ar[d]_{{\alpha_0}_{|}}\ar[r]^{} & \Cu_{\overline{\K}_1}(\mathfrak{p}_j) \ar[d]_{{\alpha}_{|}}\ar[r]^{} & \overline{\K}_1(\mathfrak{p}_j) \ar[d]^{{\alpha_{\mathrm{max}}}_{|}}\ar[r]^{} & 0
\\
0\ar[r]^{} & \Cu(\mathfrak{q}_j)\ar[r]_{} & \Cu_{\overline{\K}_1}(\mathfrak{q}_j)\ar[r]_{} & \overline{\K}_1(\mathfrak{q}_j)\ar[r]^{} & 0
} 
\] 
where the vertical arrows are isomorphisms in their respective categories obtained via restriction. 

In particular, we obtain that 
\[
\overline{\K}_1(\mathfrak{p}_j)\simeq \overline{\K}_1(\mathfrak{q}_j), \text{ for any } j\geq 1.
\]
As a result, the next diagram commutes.
\[
\xymatrix{
\overline{\K}_1(\mathfrak{p}_j)\ar[r]^{\simeq}\ar[d]_{\overline{\K}_1(\iota_j^A)}& \overline{\K}_1(\mathfrak{q}_j)\ar[d]^{\overline{\K}_1(\iota_j^B)}\\
 \overline{\K}_1(A)\ar[r]^{\simeq}& \overline{\K}_1(B)
 }
\]

Nevertheless, \autoref{thm:diago} tells us that $\{\overline{\K}_1(\iota_j^A)\}_j$ are simultaneously Nielsen-Thomsen diagonalisable, while $\{\overline{\K}_1(\iota_j^B)\}_j$ are not, which leads to a contradiction.
\end{proof}

\subsection{The Robert example(s) revisited} In \cite{C25}, the author has exposed a pair of non-unitarily equivalent *-homomorphisms from $C(\T)$ into $C[0,1]\otimes M_{2^\infty}$. (These were based on private communication with L. Robert.) 
The original distinction relied on the information given by de la Harpe-Skandalis determinant of the identity map. We intend to generalize these constructions and (re)state similar results via our methods. We also obtain additional information on this family, by measuring how far these *-homomorphisms are from one another, with respect to the metric $\mathfrak{d}$.

We start by recalling the construction of the two *-homomorphisms exposed in \cite[Section 4.A]{C25}. 

$\bullet$\,\,\textbf{Construction of the unitary elements of $C[0,1]\otimes M_{2^\infty}$.} 
Recall that $M_{2^\infty}$ can be written as the inductive limit of $(M_{2^n},\phi_{nm})_n$ where $\phi_{nn+1}:M_{2^n}\longrightarrow M_{2^{n+1}}$ sends $a\longmapsto \begin{psmallmatrix}a \\&a \end{psmallmatrix}$. 

For any $n\in \N$ we consider the following unitary element of $M_{2^n}$ 
\[w_n:= \begin{psmallmatrix}1 \\&e^{2i\pi/2^n} \\&&\ddots\\&&&e^{2i\pi(2^n-1)/2^n}\end{psmallmatrix}.\]  
It can be argued that the sequence $(\phi_{n\infty}(w_n))_n$ converges towards a unitary $w\in M_{2^\infty}$ with full spectrum. (See \cite[4.A]{C25}.) Finally, we define unitary elements of $C[0,1]\otimes M_{2^\infty}$ for any $k\in\N$ as follows \[u_0:=1_{[0,1]}\otimes w \quad\quad\text{and}\quad\quad u_k:=e^{2i\pi k\id_{[0,1]}}\otimes w.\] 
 We also consider the induced *-homomorphisms $\varphi_{u_k}\colon C(\T)\longrightarrow C[0,1]\otimes M_{2^\infty}$, mapping $\id_\T\mapsto u_k$.
\begin{thm} 
We compute that \[
d^*_{\Cu}(\Cu_{\K_1}(\varphi_{u_k}),\Cu_{\K_1}(\varphi_{u_l}))=0\quad \text{ and }\quad \mathfrak{d}(\overline{\K}_1(\varphi_{u_k}),\overline{\K}_1(\varphi_{u_l}))=\frac{\vert k-l\vert}{2}.
\]
\end{thm}

\begin{proof}
Let us fix $A:=C[0,1]\otimes M_{2^\infty}$. Let $k,l\in \N$ and $n\in\N$. We consider $u_{k,n}:=e^{2i\pi k\id_{[0,1]}}\otimes w_n$ and $u_{l,n}:=e^{2i\pi l\id_{[0,1]}}\otimes w_n$. We proceed similarly as in the proof of \autoref{lma:computeCu}, to compute that \[d_{\Cu}(\Cu(\varphi_{u_{k,n}}),\Cu(\varphi_{u_{l,n}}))\leq 1/2^n.\]
(We refer the reader to \cite[4.A]{C25} for an explicit computation.)
Now, let $U\subseteq \T$ be an open set.  It is readily observed that the ideals generated by $\Cu(\varphi_{u_{k,n}})(\mymathbb{1}_{U_{\frac{1}{2^n}}})$ and $\Cu(\varphi_{u_{l,n}})(\mymathbb{1}_{U_{\frac{1}{2^n}}})$ are in fact . Combined with the fact that $A$ has trivial $\K_1$-group, we deduce that the fiber diagram of $(\Cu(\varphi_{u_{k,n}}),\Cu(\varphi_{u_{l,n}}))$ at coordinates $(\mymathbb{1}_U,\mymathbb{1}_{U_{\frac{1}{2^n}}})$ trivially commutes. This yields
\[
d^*_{\Cu}(\Cu_{\K_1}(\varphi_{u_{k,n}}),\Cu_{\K_1}(\varphi_{u_{l,n}}))\leq 1/2^n.
\]
We remark that for any $k\in\N$, $d_U((\id\otimes \phi_{n\infty})(u_{k,n}),u_k)\underset{n\rightarrow \infty}{\longrightarrow}0$. We now deduce the first computation by a standard argument.

Next, let us compute the distance $\mathfrak{d}$ between $\overline{\K}_1(\varphi_{u_k})$ and $\overline{\K}_1(\varphi_{u_l})$. It is immediate that $\K_1(\varphi_{u_k})=\K_1(\varphi_{u_l})$ is the trivial morphism. Also, we know that $d_{\Cu}(\Cu(\varphi_{u_k}),\Cu(\varphi_{u_l}))=0$, which implies that $d(\HH(\varphi_{u_k}),\HH(\varphi_{u_l}))=0$, by \autoref{thm:Cufinernose}. Let $\mathcal{C}_0$ be the canonical Nielsen-Thomsen basis of $C(\T)$ and let $\mathcal{D}$ be any Nielsen-Thomsen basis of $A$. 

From \autoref{lma:key}, we get that $R_{\mathcal{C}_0\mathcal{D}}(\varphi_{u_k})-R_{\mathcal{C}_0\mathcal{D}}(\varphi_{u_l})\colon \Z\longrightarrow \HH(A)$ sends $1_\Z\longmapsto \overline{\Delta}\begin{psmallmatrix} u_k & 0 \\
0 & u_l
\end{psmallmatrix}=\overline{\Delta}(u_k)-\overline{\Delta}(u_l)$. Furthermore, it is well-known that $\Aff T_1(A)\simeq C([0,1],\R)$, and that $\K_0(A)\simeq \N[\frac{1}{2}]$. We deduce that \[\HH(A)\simeq C([0,1],\R)/\{\text{constant functions}\}.\] 

Finally, standard arguments show that the quotient-norm in $C([0,1],\R)/\{\text{constant functions}\}$, is given by $\Vert [f]\Vert=\frac{1}{2}(\max f(t)-\min f(t))$, for any $f\in C([0,1],\R)$. 

Now, let $C:=\tau_M(h_w)\in \R$, where $\tau_M$ is the unique trace on $M_{2^\infty}$ and $h_w\in M_{2^\infty}$ is any self-adjoint element such that $e^{2i\pi h_w}=w$. The computations done in \cite[4.A]{C25} give us that 
\[
\overline{\Delta}(u_k)=[t\longmapsto C+kt]_{\HH(A)}
\]
for any $k\in\N$. The second computation readily follows.
\end{proof}

\begin{cor}
The *-homomorphisms $\{\varphi_{u_k}\colon C(\T)\rightarrow C[0,1]\otimes M_{2^\infty}\}_k$ all agree on the Cuntz semigroup and the unitary Cuntz semigroup. 

Yet, they are pairwise not approximately unitarily equivalent. More particularly, for any distinct $k,l\in \N$, $\varphi_{u_k}$ and $\varphi_{u_l}$ are distinguished by the Hausdorffized algebraic $\K_1$-group. 

(A fortiori, by the Hausdorffized  unitary Cuntz semigroup.)
\end{cor}

\subsection{A novel example}  We end the manuscript with a novel example that illustrates the necessity of the Hausdorffized unitary Cuntz semigroup in order to classify *-homomorphisms from $C(\T)$. We exhibit a pair of unitary elements $u,v$ of an $\AI$-algebra $A$, whose induced *-homomorphisms $\varphi_u,\varphi_v$ are shown to agree on both the unitary Cuntz semigroup and the Hausdorffized algebraic $\K_1$-group. (This, in turn, implies that they coincide on the Cuntz semigroup and the $\K_1$-group.)  Nevertheless, $\varphi_u$ and $\varphi_v$ are not approximately unitarily equivalent, as they are distinguished by the Hausdorffized unitary Cuntz semigroup and thereby, showcases its ability to distinguish *-homomorphisms that cannot be separated by any of the other invariants considered in this study.

$\bullet$\,\,\textbf{Construction of $A$ and its unitary elements.} Let us consider $A$ to be the unital $\AI$-algebra obtained as the inductive limit of $(C[0,1]\otimes M_{2^n},\phi_{nn+1})_{n\in\N}$, where $\phi_{nn+1}\colon C[0,1]\otimes M_{2^n}\longrightarrow C[0,1]\otimes M_{2^{n+1}}$ sends $f\mapsto \begin{psmallmatrix} f \\& f(0) \end{psmallmatrix}$.

 Let us consider two piecewise-linear functions $f,g\colon [0,1]\longrightarrow \R$, as follows. 

(i) Set $f(0)=0$, $f(\frac{1}{2})=\frac{1}{4}$, $f(1)=0$, and define $f$ to be linear between these points.

(ii) Set $g(0)=0$, $g(\frac{1}{2})=0$, $g(1)=0$, $g(\frac{1}{4})=\frac{1}{4}$, $g(\frac{3}{4})=\frac{1}{4}$, and define $g$ to be linear between these points.

For any $n\in\N$, we consider the following diagonal unitary elements of $C[0,1]\otimes M_{2^n}$.
\[
u_n:= \diag(
e^{2i\pi f}, \lambda_{1,n}, \dots, \lambda_{2^{n-1}-1,n}, e^{i\pi}e^{2i\pi g}, e^{i\pi}\lambda_1, \dots, e^{i\pi}\lambda_{2^{n-1}-1,n} )\quad\quad \text{and} \quad\quad v_n:=e^{i\pi}u_n
\]  
where $\lambda_{k,n}:=(e^{2i\pi k/2^{n-1}})^{\frac{1}{4}}$, for any $1\leq k\leq 2^{n-1}-1$.

As in the previous example, it can be argued that $d_U(\phi_{nm}(u_n),u_m)\leq 1/2^{n-1}-1/2^{m-1}$. (See also the original argument in \cite[4.A]{C25}. Roughly speaking, both $\phi_{nm}(u_n)$ and $u_m$ have the same two \textquoteleft moving\textquoteright\ eigenvalues, and their \textquoteleft fixed\textquoteright\ eigenvalues can be paired to be at distance at most $1/2^{n-1}-1/2^{m-1}$.) 
Consequently, the sequence $(\phi_{n\infty}(u_n))_n$ is Cauchy for $d_U$ and hence, converges towards a unitary element $u$ in $A$. 
Similarly, the sequence $(\phi_{n\infty}(v_n))_n$ converges towards a unitary element $v$ in $A$. We observe that $\spectrum(u)=\spectrum(v)=[1,e^{2i\pi/4}] \cup [e^{i\pi},e^{i\pi}e^{2i\pi/4}]\subseteq \T$. 

\begin{thm} 
We compute that \[
\left\{\begin{array}{ll}
d^*_{\Cu}(\Cu_{\K_1}(\varphi_{u}),\Cu_{\K_1}(\varphi_{v}))=0\\
\mathfrak{d}(\overline{\K}_1(\varphi_{u}),\overline{\K}_1(\varphi_{v}))=0\\
\mathfrak{d}^*_{\Cu}(\Cu_{\overline{\K}_1}(\varphi_{u}),\Cu_{\overline{\K}_1}(\varphi_{v}))\geq 1/8
\end{array}
\right.
\]

As a consequence, $\varphi_u$ and $\varphi_v$ are not approximately unitarily equivalent even though they agree on the unitary Cuntz semigroup and the Hausdorffized algebraic $\K_1$-group.
\end{thm}

\begin{proof}
We proceed similarly as in the previous example and the proof of \cite[4.A]{C25} to compute that \[d_{\Cu}(\Cu(\varphi_{u_{n}}),\Cu(\varphi_{v_{n}}))\leq 1/2^{n-1}.\]

Let $U\subseteq \T$ be an open set. We write $V:=U_{\frac{1}{2^{n-1}}}$. The ideals generated by $\Cu(\varphi_{u_{n}})(\mymathbb{1}_V)$ and $\Cu(\varphi_{v_{n}})(\mymathbb{1}_V)$ are either trivial or equal to $A$, and hence have trivial $\K_1$-groups. This implies that 
\[
d^*_{\Cu}(\Cu_{\K_1}(\varphi_{u_{n}}),\Cu_{\K_1}(\varphi_{v_{n}}))\leq 1/2^{n-1}.
\]
The first computation now follows from a standard argument. 

Secondly, observe that $u=e^{i\pi}v$ and hence, $\overline{\Delta}(u)=\overline{\Delta}(v)$. Combined with the fact that $\K_1(\varphi_u)$ and $\K_1(\varphi_v)$ are trivial and that $d(\HH(\varphi_u),\HH(\varphi_v))=0$ (since $d_{\Cu}(\Cu(\varphi_u),\Cu(\varphi_v))=0$, see \autoref{thm:Cufinernose}), we get that 
\[
\mathfrak{d}(\overline{\K}_1(\varphi_{u}),\overline{\K}_1(\varphi_{v}))=0.
\]

Lastly, let us write $\alpha:=\Cu_{\overline{\K}_1}(\varphi_{u})$ and $\beta:=\Cu_{\overline{\K}_1}(\varphi_{v})$. Similarly, we write $\alpha_n:=\Cu_{\overline{\K}_1}(\varphi_{u_n})$ and $\beta_n:=\Cu_{\overline{\K}_1}(\varphi_{v_n})$. Recall that $\alpha=(\alpha_0,\{\alpha_I\}_{I\in\Lat_f(A)})$, where $\alpha_0:=\Cu(\varphi_u)$ and $\alpha_I:=\overline{\K}_1(I\overset{\varphi_u}{\longrightarrow} I_{\varphi_u})$. Similarly $\beta=((\beta_0,\{\beta_I\}_{I\in\Lat_f(A)}))$ and $\alpha_n,\beta_n$ are also of this form. 

We denote $I$ to be the ideal of $C(\T)$ generated by $x:=\mymathbb{1}_{(1,e^{2i\pi/4})}\in \Lsc(\T,\overline{\N})$. Observe that $I_{\varphi_u}=I_{\varphi_v}=I_{\varphi_{u_n}}=I_{\varphi_{v_n}}=A$. Therefore, for any $y\in \Lsc(\T,\overline{\N})$ such that $x\ll y$, the fiber diagram of $(\alpha,\beta)$ at coordinates $(x,y)$ falls down to 
\[
\xymatrix{
\overline{\K}_1(I)\ar@<0,5ex>[r]^{\alpha_{I}}\ar@<-0,5ex>[r]_{\beta_I} &\overline{\K}_1(A).
}
\]
Similarly for the fiber diagram of $(\alpha_n,\beta_n)$ at coordinates $(x,y)$. 

We aim to compute a (non-zero) lower-bound $m>0$ for $\mathfrak{d}(\alpha_{I},\beta_{I})$. 
First, observe that for any $\epsilon>0$, there is $n\in \N$ big enough such that $d_U(\varphi_{u_n},\varphi_u),d_U(\varphi_{v_n},\varphi_v)<\epsilon$. Additionally, we have
\begin{align*}
\mathfrak{d}((\alpha_n)_I,(\beta_n)_I)&\leq \mathfrak{d}(\overline{\K}_1(\phi_{n\infty})\circ(\alpha_n)_I, \overline{\K}_1(\phi_{n\infty})\circ(\beta_n)_I)\\
&\leq \mathfrak{d}(\alpha_I,\overline{\K}_1(\phi_{n\infty})\circ(\alpha_n)_I) +\mathfrak{d}(\alpha_I,\beta_I)+ \mathfrak{d}(\beta_I,\overline{\K}_1(\phi_{n\infty})\circ(\beta_n)_I)\\
&\leq 2\epsilon+\mathfrak{d}(\alpha_I,\beta_I).
\end{align*}

Therefore, it is enough to find a uniform lower-bound for $\mathfrak{d}((\alpha_n)_I,(\beta_n)_I)$, for any $n\in\N$. We proceed similarly as in the previous example, to compute that $\overline{\Delta}((u_n)_I)=[t\mapsto C+ 4f]_{\HH(A)}$ and $\overline{\Delta}((v_n)_I)=[t\mapsto C'+ 4g]_{\HH(A)}$, for some constants $C,C'\in \R$. We finally deduce that $\mathfrak{d}((\alpha_n)_I,(\beta_n)_I)\geq \frac{1}{2}(\max 4(f-g)(t)-\min 4(f-g)(t))= 1/2$, and hence, that $\mathfrak{d}(\alpha_I,\beta_I)\leq 1/2$. By \autoref{prop:lwrbnd}, we conclude that $d^*_{\Cu,\mathfrak{d}}(\Cu_{\overline{\K}_1}(\varphi_{u}),\Cu_{\overline{\K}_1}(\varphi_{v}))\geq 1/8$. 
\end{proof}

\begin{qst} At the end of \cite{C25}, it is conjectured that the Hausdorffized unitary Cuntz semigroup could classify *-homomorphisms from $\C(\T)$ to a large class of $\CatCa$-algebras, containing $\AI$ and $\A\!\T$-algebras. 

As a first step towards this conjecture, we ask whether the above techniques and computations could yield a classification of *-homomorphisms from $C(\T)$ to $C([0,1])\otimes A$, where $A$ is any $\UHF$ algebra, by means of the Hausdorffized unitary Cuntz semigroup? 
\end{qst}

\end{document}